\documentclass [11pt, a4paper, twoside, reqno]{amsart}

\usepackage[PS]{diagrams}
\usepackage{enumerate}
\usepackage{hyperref}

\numberwithin{equation}{section}

\theoremstyle{plain}
\newtheorem{theorem}{Theorem}[subsection]
\newtheorem{proposition}[theorem]{Proposition}
\newtheorem{lemma}[theorem]{Lemma}
\newtheorem{corollary}[theorem]{Corollary}
\newtheorem*{main}{Theorem~\ref{thm:main}}
\newtheorem{assertion}[equation]{}

\theoremstyle{definition}
\newtheorem{definition}[theorem]{Definition}
\newtheorem{remark}[theorem]{Remark}

\newtheorem{numpar}[theorem]{}

\numberwithin{equation}{theorem}

\newcommand{\nbd}{\nobreakdash}
\newcommand{\cone}{\mathrm{cone}\,}
\newcommand{\bR}{{\mathbb R}}
\newcommand{\bZ}{{\mathbb Z}}
\newcommand{\bC}{{\mathbb C}}
\newcommand{\bN}{{\mathbb N}}
\newcommand{\iso}{\cong}
\newcommand{\tensor}{\otimes}
\newcommand{\ie}{{\it i.e.}}
\newcommand{\pre}{\mathrm{Pre}\,}
\newcommand{\qco}{\mathrm{qCoh}\,}
\newcommand{\hco}{\mathrm{hCoh}\,}
\newcommand{\hvect}{\mathrm{hVect}\,}
\newcommand{\perf}{\mathrm{Perf}\,}

\newcommand{\Ch}{\mathrm{Ch}\,}
\newcommand{\RMod}{R\text{-}\mathrm{Mod}}
\newcommand{\con}{\mathrm{con}\,}

\date{}

\begin{document}

\title[$K$\nbd-theory of projective toric schemes]{A splitting result
  for the algebraic $K$\nbd-theory of projective toric schemes}

\author{Thomas H\"uttemann}

\address{Thomas H\"uttemann, Queen's University Belfast, Pure
  Mathematics Research Centre, Belfast BT7~1NN, Northern Ireland, UK}

\email{t.huettemann@qub.ac.uk}

\urladdr{http://huettemann.zzl.org/}

\subjclass[2000]{Primary 18F25; Secondary 14F05, 18F20, 18G35}

\keywords{Toric scheme, quasi-coherent sheaf, perfect complex,
  algebraic $K$-theory}

\thanks{This work was supported by the Engineering and Physical
  Sciences Research Council [grant number EP/H018743/1].}

\begin{abstract}
  Suppose $X$ is a projective toric scheme defined over a ring~$R$ and
  equipped with an ample line bundle~$\mathcal{L}$. We prove that its
  $K$-theory has a direct summand of the form $K(R)^{k+1}$ where $k
  \geq 0$ is minimal such that $\mathcal{L}^{\tensor (-k-1)}$ is {\it
    not\/} acyclic. Using a combinatorial description of
  quasi-coherent sheaves we interpret and prove this result for a
  ring~$R$ which is either commutative, or else left
  \textsc{noether}ian.
\end{abstract}

\maketitle

{\small \tableofcontents
}\goodbreak

\section*{Introduction}

Let~$R$ be a commutative ring with unit. An $n$\nbd-dimensional
polytope $P \subseteq \bR^n$ with integral vertices determines a
projective toric $R$\nbd-scheme~$X_P$ together with a family of line
bundles~$\mathcal{O}_{X_P}(k)$ (a family of quasi-coherent sheaves of
$\mathcal{O}_{X_P}$\nbd-modules which are locally free of
rank~$1$). To 
a given chain complex of quasi-coherent sheaves~$\mathcal{F}_\bullet$
we associate an $R$\nbd-module chain complex $\check \Gamma
(\mathcal{F}_\bullet)$ given by first forming a \textsc{\v Cech}
complex in each chain degree of~$\mathcal{F}_\bullet$, and then taking
the total complex of the resulting twofold chain complex.

Set $n_P = 0$ if $P$~has integral points in its interior; otherwise,
let $n_P \geq 1$ be such that the dilated polytope $(n_P+1) P$ has
lattice points in its interior, but $n_P P$~does not. This number can
be characterised in different ways: $n_P \geq 0$ is minimal among
non-negative integers~$k$ such that $\mathcal{O}_{X_P} (-k-1)$ is not
acyclic; also, $n_P$~is the number of distinct integral roots of the
\textsc{Ehrhart} polynomial of~$P$.

We will show that there is a homotopy equivalence of $K$\nbd-theory
spaces
\[K(X_P) \simeq \underbrace{K(R) \times \ldots \times K(R)}_{n_P+1
  \text{ factors}} \times K(X_P,\, [n_P])\]where $K$\nbd-theory is
defined using perfect complexes of sheaves and modules, and the last
factor on the right denotes the $K$\nbd-theory of those perfect
complexes $\mathcal{F}_\bullet$ of quasi-coherent sheaves on~$X_P$ for
which $\check \Gamma \big( \mathcal{F}_\bullet (k) \big)$ is acyclic
for $0 \leq k \leq n_P$. In fact, we will prove slightly more: by
exploiting a strictly combinatorial viewpoint of sheaves on toric
varieties we can prove the corresponding result for a unital ring~$R$
which is commutative, or else left \textsc{noether}ian.

A corresponding result has been proved by the author in a
``non-linear'' context, replacing modules by topological
spaces~\cite{H-nonlin}.  It must also be pointed out that the
splitting result is, in general, far from optimal: a lot of
$K$\nbd-theoretical information can be left over in the factor
$K(X_P,\, [n_P])$. For example, if $n_P = 0$ (which can be guaranteed
by first replacing~$P$ by its dilate $(n+1)P$) the splitting results
merely gives a version of reduced $K$\nbd-theory. But in the other
extreme, if $P$~is an $n$\nbd-dimensional simplex with volume $1/n!$
then $X_P = \mathbb{P}^n$ is $n$\nbd-dimensional projective space; one
can show that $n_P = n$ and $K(\mathbb{P}^n,\, [n]) \simeq *$ in this
case so that we recover the known splitting results for projective
spaces, generalised to suitable non-commutative ground rings
(Theorem~\ref{thm:k_of_Pn}). There are, of course, intermediate cases;
suffice it to mention that $0/1$\nbd-polytopes do not have interior
lattice points and hence lead to non-trivial splittings with $n_P \geq
1$. Interesting examples of~$0/1$\nbd-polytopes arise as
\textsc{Stanley}'s order and chain polytope associated to finite
posets~\cite{Stanley-popo}.

The paper is divided into three parts. In \S1 we introduce the
combinatorial framework for sheaves on projective toric varieties, and
give a first formulation of the main result,
Theorem~\ref{thm:main}. We also prove that for a {\it commutative\/}
ring~$R$ we recover the usual notions of algebraic geometry. In \S2 we
develop some algebraic geometry from the combinatorial viewpoint,
allowing for a non-commutative ground ring~$R$: we define twisting
sheaves and study the \textsc{\v Cech\/} complex of various of
complexes of sheaves. Of major importance is the finiteness
Theorem~\ref{thm:Cech_finite} which asserts that the \textsc{\v
  Cech\/} complex of a perfect complex is a perfect complex of
$R$\nbd-modules. In the left \textsc{noether}ian case this is quite
straightforward, while the non-\textsc{noether}ian commutative case
requires \textsc{noether}ian approximation (descent). Finally,
\S3~contains a detailed formulation of the main theorem, and its
proof.

We assume some familiarity with basic homological algebra as presented
by \textsc{Weibel} \cite{Weibel-intro}, \textsc{Waldhausen\/}
$K$\nbd-theory \cite{W-Ktheory} and its formulation in an
algebro-geometric setting by \textsc{Thomason} and \textsc{Trobaugh}
\cite{TT-K}. We mention a few conventions used in this paper. Chain
complexes are topologically indexed: differentials lower the degree
by~$1$. If needed, modules are considered as chain complexes
concentrated in chain degree~$0$. The term ``module'' without any
qualification refers to a left module.

\section{The $K$-theory of projective toric varieties}

\subsection{Complexes of sheaves on projective toric varieties}

We start by a combinatorial description of quasi-coherent sheaves on
toric schemes defined by a polytope.

Let $P \subseteq \bR^n$ be an $n$\nbd-dimensional polytope with
integral vertices. Each non-empty face~$F$ of~$P$ gives rise to a cone
(the set of finite linear combinations with non-negative real
coefficients)
\[T^P_F = T_F = \cone \{p-f \,|\, p \in P,\, f \in F\} \ ,\] the {\it
  tangent cone of~$F$}, and hence to an additive monoid $S^P_F = S_F =
T_F \cap \bZ^n$.

\begin{lemma}
  Let $F \subseteq G$ be a pair of non-empty faces of~$P$. Then there
  exists a vector $v_{F,G} \in \bZ^n \cap T_F$ such that
  \[T_G = T_F + \bR v_{F,G} \qquad \text{and} \qquad S_G = S_F + \bZ
  v_{F,G} \ .\] In the former case the symbol~``$+$'' denotes
  \textsc{Minkowski} sum, the set of sums of elements of the indicated
  cones; in the latter, the symbol~``$+$'' denotes the submonoid
  of~$\bZ^n$ generated by sums of elements of the indicated monoids.
\end{lemma}

\begin{proof}
  This is Proposition~1.2.2 of~\cite{Fulton:TVs}, applied to the dual
  cones $\sigma = T_F^\vee$ and $\tau = T_G^\vee$.
\end{proof}

Now let $R$~denote an arbitrary ring with unit. We then have monoid
$R$\nbd-algebras $A^P_F = A_F = R[S_F]$ for every non-empty face~$F$
of~$P$; an element $v \in S_F$ gives rise to the element $x = 1 \cdot
v \in A_F$. Note that $A_F$~is not commutative unless $R$~is, but that
the element~$x$ lies in the centre of~$A_F$. Moreover, monoid
generators of~$S_F$ give rise to $R$\nbd-algebra generators of~$A_F$
which is thus a finitely generated algebra by \textsc{Gordan}'s Lemma
(\cite[1.2.1]{Fulton:TVs}, applied to the dual cone $\sigma =
T_F^\vee$). The previous Lemma implies immediately:

\begin{lemma}
  \label{lem:F_G_loc}
  For $F \subseteq G$ the algebra~$A_G$ is obtained from the
  algebra~$A_F$ by localising by a single element $x_{F,G}$ in the
  centre of~$A_F$. Consequently, for every $p \in A_G$ there is $N
  \geq 0$ such that $x_{F,G}^N \cdot p \in A_F$.  \qed
\end{lemma}

Replacing~$P$ by an integral dilate~$kP = \{kp \,|\, p \in P\}$,
$k \geq 1$, and replacing the face~$F$ by its dilate~$kF$ does not
change these cones, monoids and algebras. That is, $T^P_F =
T^{kP}_{kF}$, and similarly for~$S_F$ and~$A_F$.

\begin{definition}[Presheaves]
  \label{def:presheaves}
  Let $F(P)_0$ denote the set of non-empty faces of~$P$, partially
  ordered by inclusion. We write $\Ch(R)$ for the category of
  (possibly unbounded) chain complexes of $R$\nbd-modules.
  \begin{enumerate}
  \item A {\it presheaf on~$P$\/} is a functor
    \[Y \colon F(P)_0 \rTo \Ch(R), \ F \mapsto Y^F\] equipped with
    extra data which turns the object $Y^F$ into a chain complex of
    (left) $A_F$-modules, such that for each pair of non-empty faces
    $F \subseteq G$ of~$P$ the structure map $Y^F \rTo Y^G$ is a map
    of $A_F$\nbd-module chain complexes.
  \item A {\it map of presheaves\/} is a natural transformation of
    functors such that its $F$\nbd-component is an $A_F$\nbd-linear
    map, for each $F \in F(P)_0$. The category of presheaves is
    denoted by $\pre (P)$.
  \item A map $f \colon Y \rTo Z$ of presheaves is called a {\it
      quasi-isomorphism}, or a {\it weak equivalence}, if all its
    components~$f^F$ are quasi-isomorphisms of chain complexes of
    modules.
  \end{enumerate}
\end{definition}

\begin{definition}[Quasi-coherent sheaves]
  \label{def:qc_sheaf}
  \begin{enumerate}
  \item An object $Y \in \pre(P)$ is called a {\it quasi-coherent
      sheaf}, or {\it sheaf\/} for short, if for each pair of
    non-empty faces $F \subseteq G$ of~$P$, the adjoint structure map
    \[A_G \tensor_{A_F} Y^F \rTo Y^G\] is an isomorphism of
    $A_G$-module chain complexes.
  \item The full subcategory of~$\pre(P)$ consisting of sheaves is
    denoted $\qco(P)$.
  \end{enumerate}
\end{definition}

\begin{definition}[Homotopy sheaves]
  \label{def:homoty_sheaf}
  \begin{enumerate}
  \item A presheaf $Y \in \pre(P)$ is called a {\it homotopy sheaf\/}
    if for each pair of non-empty faces $F \subseteq G$ of~$P$ the
    adjoint structure map $A_G \tensor_{A_F} Y^F \rTo Y^G$ is a
    quasi-isomorphism of chain complexes.
  \item The full subcategory of $\pre(P)$ consisting of homotopy
    sheaves is denoted $\hco(P)$.
  \end{enumerate}
\end{definition}

Every sheaf is a homotopy sheaf, and $\qco(P)$ is a full subcategory
of~$\hco(P)$. Moreover, the notion of a homotopy sheaf is homotopy
invariant in the following sense: {\it If $f \colon Y \rTo Z$ is a
  weak equivalence of presheaves, then $Y$~is a homotopy sheaf if and
  only if $Z$~is a homotopy sheaf.}  This is true since $A_G$~is a
localisation of~$A_F$ by \ref{lem:F_G_loc}, for every pair of
non-empty faces $F \subseteq G$ of~$P$, so that the functor $A_G
\tensor_{A_F} \,\hbox{-}\,$ is exact.

\medskip

A chain complex of modules over some ring is called {\it strict
  perfect\/} if it is bounded and consists of finitely generated
projective modules. It is called {\it perfect\/} if it is
quasi-isomorphic to a strict perfect complex. In fact, a complex~$C$
is perfect if and only if there exists a strict perfect complex~$B$
and a quasi-isomorphism $B \rTo C$.

\begin{definition}[Perfect complexes and vector bundles]
  \label{def:perfect}
  \begin{enumerate}
  \item The ho\-mo\-topy sheaf $Y \in \hco(P)$ is called a {\it
      perfect complex\/} if for each $F \in F(P)_0$ the chain complex
    $Y^F$ is a perfect complex of $A_F$\nbd-modules.
  \item The full subcategory of~$\hco(P)$ consisting of perfect
    complexes is denoted by $\perf(P)$.
  \item A homotopy sheaf $Y \in \hco(P)$ is called a {\it homotopy
      vector bundle\/} if for each $F \in F(P)_0$ the chain
    complex~$Y^F$ is a strict perfect complex of $A_F$\nbd-modules.
  \item The full subcategory of~$\hco(P)$ consisting of homotopy
    vector bundles is denoted $\hvect(P)$.
  \end{enumerate}
\end{definition}

The notion of a perfect complex is homotopy invariant in the manner
described above for homotopy sheaves.

\medskip

In algebraic geometry, a perfect complex can be replaced up to
quasi-isomorphism by a bounded chain complex of vector bundles
provided the scheme under consideration has an ample line bundle (as
is the case for a projective scheme) or, more generally, has an ample
family of line bundles. In the homotopy world the replacement is
possible without reference to such additional structure.

\begin{lemma}
  \label{lem:perf_qis_hvect}
  For $Y \in \perf(P)$ there is a homotopy vector bundle $Z \in
  \hvect(P)$ together with a quasi-isomorphism $\zeta \colon Z
  \rTo^\simeq Y$.
\end{lemma}

\begin{proof}
  By definition of perfect complexes there is, for each $F \in
  F(P)_0$, a bounded chain complex~$V^F$ of finitely generated
  $A_F$\nbd-modules together with a quasi-isomorphism $\nu^F \colon
  V^F \rTo Y^F$.

  Let $\mathcal{P}(d)$ denote the sub-poset of~$F(P)_0$ of faces of
  dimension at most~$d$. For $F$~a vertex of~$P$ define $Z^F = V^F$
  and $\zeta^F = \nu^F$; this defines a
  $\mathcal{P}(0)$\nbd-dia\-gram~$Z$ and a quasi-isomorphism $Z \rTo
  Y|_{\mathcal{P}(0)}$.

  Suppose we have constructed a quasi-isomorphism of
  $\mathcal{P}(d-1)$\nbd-diagrams $Z \rTo Y_{\mathcal{P}(d-1)}$ such
  that each component of~$Z$ is a strict perfect complex over the
  appropriate algebra. We show how to extend this data
  to~$\mathcal{P}(d)$.

  Fix a $d$\nbd-dimensional face~$F$ of~$P$, and let
  \[L^FZ = \mathop{\mathrm{colim}}_{\substack{G \subseteq F\\\dim
      G<d}} A_F \tensor_{A_G} Z^G \ .\] This is a strict perfect
  complex of $A_F$\nbd-modules. We define $L^FY$ by a similar colimit
  with~$Z^G$ replaced by~$Y^G$. These come equipped with canonical
  maps $L^FZ \rTo L^FY \rTo Y^F$ induced by the maps $\zeta^G$ defined
  before, and the structure maps of~$Y$. Up to homotopy, the
  composition factors over the quasi-isomorphism~$\nu^F$ (this follows
  from \cite[10.4.7]{Weibel-intro} together with the fact that $\nu^F$
  induces an isomorphism of hom-sets $\hom(L^FZ, V^F) \iso
  \hom(L^FZ,Y^F)$ in the derived category of~$A_F$). Let $Z^F$~be the
  mapping cylinder of $L^FZ \rTo V^F$; a homotopy then determines a
  map $\zeta^F \colon Z^F \rTo Y^F$ such that the two compositions
  \[L^FZ \rTo Z^F \rTo Y^F \qquad \text{and} \qquad L^FZ \rTo L^FY
  \rTo Y^F\] agree. Since $Z^F \simeq V^F$ the map~$\zeta^F$ is a
  quasi-isomorphism. For $G \subset F$ define a structure map $Z^G
  \rTo Z^F$ as the composition $Z^G \rTo L^F \rTo Z^F$, considered as
  maps of $A_G$\nbd-modules.

  Performing this construction for each $d$\nbd-dimensional face
  of~$P$ yields a $\mathcal{P}(d)$\nbd-diagram~$Z$ together with a
  quasi-isomorphism $\zeta \colon Z \rTo Y|_{\mathcal{P}(d)}$. At the
  $n$th step we arrive at the assertion of the Lemma.
\end{proof}

\subsection{Algebraic $K$-theory}
\label{subsec:algebraic-k-theory}

The category $\perf(P)$ carries the structure of a ``complicial
bi\textsc{Waldhausen} category'' in the sense of \textsc{Thomason} and
\textsc{Trobaugh} \cite[1.2.11]{TT-K}; as ambient \textsc{abel}ian
category we choose the category~$\pre(P)$. The weak equivalences are
as defined in~\ref{def:presheaves}. The cofibrations are the
degreewise split injections $Y \rTo Z$ in $\perf(P)$ with cokernel
in~$\perf(P)$.

\begin{definition}
  The {\it algebraic $K$-theory of~$P$\/} is defined to be the
  $K$-theory space of the complicial bi\textsc{Waldhausen} category
  $\perf(P)$. In symbols,
  \[K(P) = \Omega |w \mathcal{S}_\bullet \perf(P)|\] where
  the symbol ``$w$'' denotes the subcategory of weak equivalences as
  usual.
\end{definition}

\subsection{Justification of terminology}

Let now~$R$ be a {\it commutative\/} ring with unit. The polytope~$P$
determines a projective $R$\nbd-scheme $X_P$, obtained from the affine
schemes $U_F = \mathrm{Spec}\, A_F$ by gluing $U_F$ and~$U_G$ along
their common open subscheme $U_{F \vee G}$. A chain
complex~$\mathcal{F}$ of quasi-coherent
$\mathcal{O}_{X_P}$\nbd-modules gives rise, by evaluation on open
sets~$U_F$, to a sheaf
\begin{equation}
  \label{eq:sheaf_Y_F}
  Y_\mathcal{F} \colon F \mapsto \Gamma (U_F,\,
  \mathcal{F})
\end{equation}
as defined in~\ref{def:qc_sheaf}. The categories of chain
complexes of quasi-coherent $\mathcal{O}_{X_P}$\nbd-modules and of
perfect complexes of quasi-coherent $\mathcal{O}_{X_P}$\nbd-modules in
the sense of \cite[2.2.10]{TT-K} are equivalent, via this
construction, to the categories $\qco(P)$ and $\perf(P) \cap \qco(P)$,
respectively.

\begin{lemma}
  \label{lem:replace_by_sheaf}
  Every homotopy sheaf $Y \in \hco(P)$ can be functorially replaced by
  a chain complex of quasi-coherent $\mathcal{O}_{X_P}$\nbd-modules
  $\mathcal{F}$ in such a way that $\Gamma (U_F,\, \mathcal{F})$
  and~$Y^F$ are quasi-isomorphic. More precisely, there exists a
  homotopy sheaf $\bar Y \in \hco(P)$, and there exist maps
  $Y_\mathcal{F} \rTo \bar Y \lTo Y$, cf.~(\ref{eq:sheaf_Y_F}), which
  restrict to quasi-isomorphisms of chain complexes on $F$-components
  for every $F \in F(P)_0$. Moreover, this data can be chosen to
  depend on~$Y$ in a functorial manner.
\end{lemma}

\begin{proof}
  This is the content of \cite[4.4.1]{H-colocal}. In short, the
  homotopy sheaf~$\bar Y$ is a fibrant replacement of~$Y$ with respect
  to a suitable model structure on~$\pre(P)$ (the replacement can be
  chosen functorially in~$Y$), and $\mathcal{F}$ is the limit of the
  $F(P)_0$\nbd-diagram of quasi-coherent
  $\mathcal{O}_{X_P}$\nbd-modules $F \mapsto j^F_* (\tilde Y^F)$. Here
  $j^F$~is the inclusion $U_F \subset X_P$, $j^F_*$~is push-forward
  along~$j$, and $\tilde Y^F$ is the the chain complex of
  quasi-coherent $\mathcal{O}_{U_F}$\nbd-modules associated to the
  complex of modules~$\bar Y^F$.
\end{proof}

For $Y \in \perf(P)$ the chain complex~$\mathcal{F}$ of
Lemma~\ref{lem:replace_by_sheaf} is a perfect complex in the sense of
\cite[2.2.10]{TT-K}.  Conversely, every perfect complex $\mathcal{F}$
of quasi-coherent $\mathcal{O}_{X_P}$\nbd-modules gives rise to an
object $Y_\mathcal{F} \colon F \mapsto \Gamma (U_F,\, \mathcal{F})$
of~$\perf(P)$, by \cite[2.4.3]{TT-K} applied to the affine
schemes~$U_F$.

\begin{definition}
  The algebraic $K$\nbd-theory $K(X_P)$ of~$X_P$ is the algebraic
  $K$-theory space of the complicial bi\textsc{Waldhausen} category of
  perfect complexes of quasi-coherent $\mathcal{O}_{X_P}$\nbd-modules,
  equipped with weak equivalences the quasi-isomorphisms, and
  cofibrations the degreewise split mono\-morphisms with cokernel a
  perfect complex.
\end{definition}

This is the ``right'' definition by \cite[3.6]{TT-K} since the scheme
$X_P$ is quasi-compact (a finite union of affine schemes) and
semi-separated (with open affine subschemes~$U_F$ as semi-separating
cover).

\begin{proposition}
  \label{prop:KXPisKP}
  The algebraic $K$\nbd-theory space $K(X_P)$ of~$X_P$ is homotopy
  equivalent to $K(P)$.
\end{proposition}

\begin{proof}
  The functor sending a perfect complex~$\mathcal{F}$ to its
  associated object $Y_\mathcal{F}$ of~$\perf(P)$,
  \[Y_\mathcal{F} \colon F \mapsto \Gamma (U_F,\, \mathcal{F}) \ ,\]
  is exact and induces, in view of Lemma~\ref{lem:replace_by_sheaf},
  an equivalence of derived categories. By \cite[1.9.8]{TT-K} this
  implies that $K(X_P) \simeq K(P)$ as claimed.
\end{proof}

\subsection{Lattice points and Ehrhart polynomials}

We need to introduce a classical result on counting lattice points in
polytopes. Let $P \subseteq \bR^n$ be an $n$\nbd-dimensional polytope
with integral vertices. For a non-negative integer~$k$, let $N_P(k)$
denote the number of integral points in the dilation~$kP$
of~$P$. Since $0P = \{0\} \subset \bR^n$ we have $N_P(0)=1$.

\begin{theorem}[Ehrhart Theorem]
  \label{thm:Ehrhart}
  There is a unique polynomial $E_P(x)$ with rational coefficients,
  called the \textsc{Ehrhart} polynomial of~$P$, such that $E_P(k) =
  N_P(k)$ for all non-negative integers~$k$. The polynomial $E_P(x)$
  has degree~$n$, constant term~$1$ and leading coefficient
  $\mathrm{vol}\,(P)$ (with volume normalised so that $\mathrm{vol}\,
  \big([0,1]^n\big) = 1$). Moreover, if $j$~is a negative integer then
  $(-1)^n E_P(j)$ is the number of integral points in the interior of
  the dilation~$jP$.
\end{theorem}

The reader can find a proof of this remarkable theorem in
\cite[12.16]{MS-CombAlg} or \cite[\S18]{Barvinok-latticepoints}. ---
From the geometric meaning of the \textsc{Ehrhart} polynomial we
deduce:

\begin{corollary}
  \label{cor:zeros_of_Ehrhart}
  All integral zeros of~$E_P(x)$ are negative, and the set of integral
  zeros of~$E_P(x)$ is of the form $\{-n_P,\, -n_P+1,\, \cdots,\,
  -1\}$ for some integer $n_P \in [0,n]$. The number $n_P$ is minimal
  among integers $k \geq 0$ such that $-(k+1)P$ has integral points in
  its interior.\qed
\end{corollary}

\subsection{The main result}

We can now formulate a preliminary version of the main result of this
paper.

\begin{theorem}
  \label{thm:main}
  Let $P \subseteq \bR^n$ be an $n$\nbd-dimensional polytope with
  integral vertices, and let~$n_P$ denote the number of distinct
  integral roots of its \textsc{Ehrhart\/} polynomial. Let~$R$ be a
  ring with unit. Suppose that $R$~is commutative, or else left
  \textsc{noether}ian. Then there is a homotopy equivalence of
  $K$\nbd-theory spaces
  \[K(P) \simeq K(R)^{n_P+1} \times K(P, [n_P])\] where $K(R)$~denotes
  the $K$\nbd-theory space of the ring~$R$, defined using perfect
  complexes of $R$\nbd-modules, and $K(P, [n_P])$ denotes the
  $K$\nbd-theory space of a certain subcategory of the
  category~$\perf(P)$.
\end{theorem}

In view of Proposition~\ref{prop:KXPisKP}, this implies a splitting
result for the algebraic $K$\nbd-theory of projective toric
$R$\nbd-schemes provided $R$~is a commutative ring with unit.

The proof of Theorem~\ref{thm:main} will be given in
\S\ref{sec:further-splitting}, and will contain explicit descriptions
of the homotopy equivalence and of $K(P, [n_P])$.

\section{\textsc{\v Cech\/} cohomology}

\subsection{Double complexes and the spectral sequence argument}

We work over an arbitrary unital ring~$R$. Let $D_{*,*}$ be a double
complex of $R$\nbd-modules; that is, we are given $R$\nbd-modules
$D_{p,q}$ for $p,q \in \bZ$, ``horizontal'' and ``vertical''
differentials
\[\partial^h \colon D_{p,q} \rTo D_{p-1,q} \qquad\text{and}
\qquad \partial^v \colon D_{p,q} \rTo D_{p,q-1}\] with $\partial^h
\circ \partial^h = 0$ and $\partial^v \circ \partial^v = 0$, such
that
\[\partial^h \circ \partial^v = -\partial^v \circ \partial^h \ .\] Its
{\it total complex\/} $\mathrm{Tot}\,D_{*,*}$ is a chain complex with
$\mathrm{Tot} (D_{*,*})_n = \bigoplus_{p+q=n} D_{p,q}$ and
differential $\partial = \partial^h + \partial^v$,
cf.~\cite[1.2.6]{Weibel-intro}.

Let $D^\prime_{*,*}$ be another double complex of $R$\nbd-modules. A
{\it map of double complexes\/} $f \colon D_{*,*} \rTo D^\prime_{*,*}$
is a collection of $R$\nbd-linear maps $D_{p,q} \rTo D^\prime_{p,q}$
which commute with vertical and horizontal differentials.

\begin{proposition}[The spectral sequence argument]
  \label{prop:ss_argument}
  Let $D_{*,*}$ be a double complex of $R$\nbd-modules. Suppose that
  $D_{*,*}$ is concentrated in the first $n+1$ columns so that
  $D_{p,*} = 0$ if $p<0$ or $p>n$, or that $D_{*,*}$ is concentrated in the first $n+1$ rows so that
  $D_{*,q} = 0$ if $q<0$ or $q>n$
  \begin{enumerate}
  \item There is a convergent spectral sequence
    \begin{equation}
      \label{eq:ss}
      E^1_{p,q} = H^h_q (D_{*,p}) \Longrightarrow H_{p+q}
      \mathrm{Tot}\, D_{*,*}
    \end{equation}
    with $E^2_{p,q} = H^v_p H^h_q (D_{*,*})$. Here $H^h$ denotes
    taking homology modules with respect to the horizontal
    differential~$\partial^h$, and~$H^v$ denotes taking homology
    modules with respect to the vertical differential~$\partial^v$.
  \item Suppose that $D^\prime_{*,*}$ is another double complex of
    $R$\nbd-modules concentrated in the first $n+1$ columns or
    rows. Suppose that the map of double complexes $f \colon D_{*,*}
    \rTo D^\prime_{*,*}$ induces an isomorphism on horizontal homology
    modules. Then $f$~induces a quasi-isomorphism
    \[\mathrm{Tot}\,(f) \colon \mathrm{Tot}\, D_{*,*} \rTo^\simeq
    \mathrm{Tot}\, D^\prime_{*,*} \ .\]
  \end{enumerate}
\end{proposition}

\begin{proof}
  The spectral sequence in~(1) arises in the standard way from a
  filtration of~$\mathrm{Tot}\,D_{*,*}$ by the rows of~$D_{*,*}$, see
  \cite[\S5.6]{Weibel-intro} for details of the construction, and
  \cite[5.2.5]{Weibel-intro} for convergence. The result of~(2) now
  follows from convergence of the bounded spectral sequences for
  $D_{*,*}$ and~$D^\prime_{*,*}$, together with the fact that by
  hypothesis $f$~induces an isomorphism of spectral sequences on the
  $E^1$\nbd-term.
\end{proof}

\subsection{\textsc{\v Cech} cohomology}

Assume now that we have oriented the faces of~$P$ so that we have
incidence numbers $[F:G] \in \{-1,\, 0,\, 1\}$ at our disposal.

\begin{definition}
  \label{def:cech_complex}
  \begin{enumerate}
  \item \label{item:cech1} Given a diagram $A \colon F(P)_0 \rTo
    \RMod$ we define its {\it \textsc{\v Cech\/} complex\/} to be the
    bounded chain complex $\check \Gamma_P(A) = \check \Gamma (A)$
    given by
    \[\check \Gamma (A)_s:= \bigoplus_{\dim F = n-s} A^F \ ,\] the sum
    extending over all non-empty faces of~$P$, with differentials given
    by $A^G \lTo[l>=3em]^{[F:G]} A^F$ for the pair $F \subseteq G$ of
    non-empty faces of~$P$.
  \item \label{item:cech2} Let $Y \colon F(P)_0 \rTo \Ch(R)$ be a
    diagram of chain complexes of $R$\nbd-modules. We define the {\it
      \textsc{\v Cech} complex $\check \Gamma_P (Y) = \check \Gamma
      (Y)$ of~$Y$} to be the total complex of the double chain complex
    of $R$\nbd-modules
    \begin{equation}
      \label{eq:associated_double_complex}
      D_{s,t} (Y) = D_{s,t} = \bigoplus_{\dim F = n-s} Y_t^F
    \end{equation}
    with horizontal differentials given by $Y_t^G \lTo[l>=3em]^{[F:G]}
    Y_t^F$ and vertical differential given by the differential in~$Y$
    multiplied by the sign~$(-1)^s$, cf.~\cite[1.2.5]{Weibel-intro}.
  \item \label{item:cech3} Any presheaf $Y \in \pre(P)$ can be
    considered as a diagram of chain complexes of $R$\nbd-modules, and
    we define its {\it \textsc{\v Cech\/} complex\/} $\check \Gamma_P
    (Y) = \check \Gamma(Y)$ as in~(\ref{item:cech2}).
  \end{enumerate}
\end{definition}

\begin{remark}
  \label{rem:on_check_Gamma}
  \begin{enumerate}
  \item \label{item:limits} The homology modules of~$\check \Gamma(A)$
    in~\ref{def:cech_complex}~(\ref{item:cech1}) are isomorphic to
    higher derived inverse limits of the diagram~$A$; more precisely,
    $\lim{}^k (A) \iso H_{n-k} \check \Gamma(A)$. See
    \cite[2.19]{H-finiteness} for a proof.
  \item If $A \colon F(P)_0 \rTo \RMod$ is a constant diagram with
    value $A^F = M$ for all~$F$, then $H_n \check \Gamma(A) = M$ and
    $H_k \check \Gamma(A) = 0$ for $k \ne n$. This follows easily
    from~(1), or from the observation that $\check \Gamma(A)$~is dual
    to the chain complex computing cellular homology of the
    polytope~$P$ with coefficients in~$M$, up to re-indexing.
  \item If $R$ is a commutative ring and $Y \in \qco(P)$ is
    concentrated in chain degree~$0$, then $Y$ determines a
    quasi-coherent sheaf~$\mathcal{F}$ on the scheme~$X_P$. By
    \cite[2.18]{H-finiteness} we have isomorphisms $H_{n-k} \check
    \Gamma (Y) \iso H^k (X_P,\, \mathcal{F})$, the $k$th cohomology
    module of~$X_P$ with coefficients in the sheaf~$\mathcal{F}$.
  \item For the diagram~$Y$ in
    \ref{def:cech_complex}~(\ref{item:cech2}) $D_{*,*}$ is
    concentrated in the first $n+1$ columns so that $D_{p,*} = 0$ if
    $p<0$ or $p>n$. We have $D_{*,t} = \check \Gamma (Y_t)$ which is a
    chain complex computing $\lim{}^{n-*} Y_t$. The double chain
    complex $D_{*,*} (Y)$ thus gives rise to a convergent
    (homological) spectral sequence
    \begin{equation}
      \label{eq:ss_2}
      E^1_{s,t}(Y) = \lim_\leftarrow{}^{n-s} (Y_t) \Longrightarrow
      H_{s+t} \check \Gamma (Y) \ ,
   \end{equation}
    cf.~Proposition~\ref{prop:ss_argument}~(1).
  \end{enumerate}
\end{remark}

\begin{remark}
  \label{rem:other_ss}
  There is another standard spectral sequence which we will have
  occasion to use. Let $Y \in \pre(P)$, or more generally, let $Y$~be
  a diagram $F(P)_0 \rTo \Ch(R)$. Denote by $D_{*,*}$ the double
  complex of $R$\nbd-modules associated
  to~$Y$~\eqref{eq:associated_double_complex}. Filtration by columns
  yields a convergent $E^1$\nbd-spectral sequence
  \[E^1_{p,q}(Y) = \bigoplus_{\dim F = n-p} H_q (Y^F) \Longrightarrow
  H_{p+q} \check \Gamma (Y) \ ,\] cf.~\cite[5.6.1]{Weibel-intro}; by
  Remark~\ref{rem:on_check_Gamma}~(\ref{item:limits}), $E^2_{p,q}(Y) =
  \lim\limits_{\leftarrow}{}^{n-p} H_q(Y)$.
\end{remark}

\begin{proposition}
  \label{prop:cech_invariant}
  Formation of \textsc{\v Cech\/} complexes is homotopy
  invariant. More precisely, let $f \colon Y \rTo Z$ be a map of
  $F(P)_0$\nbd-diagrams of $R$\nbd-module chain complexes. Suppose
  that for each $F \in F(P)_0$ the $F$-component of~$f$ is a
  quasi-isomorphism. Then $f$~induces a quasi-isomorphism $\check
  \Gamma(Y) \rTo \check \Gamma (Z)$.
\end{proposition}

\begin{proof}
  Consider the spectral sequence~\ref{rem:other_ss} for~$Y$
  and~$Z$. By hypothesis, the map~$f$ induces an isomorphism of
  $E^1$\nbd-spectral sequences $E^1_{*,*}(Y) \iso E^1_{*,*}(Z)$, hence
  induces an isomorphism of their abutments. But this is just a
  reformulation of the claim.
\end{proof}

\subsection{Line bundles determined by~$P$}

The polytope~$P$ determines a family of objects of~$\qco(P)$ as
follows. For $k \in \bZ$ we define
\[\mathcal{O}(k) \colon F \mapsto \mathcal{O}(k)^F = R[(kF + T_F) \cap
\bZ^n]\ ,\] considered as a diagram of chain complexes concentrated in
degree~$0$. Here $T_F$~is the tangent cone of~$P$ at~$F$, and $kF +
T_F = \{kf + v \,|\, f \in F,\, v \in T_F\}$ is the \textsc{Minkowski}
sum of the dilation~$kF$ of~$F$ and the cone~$T_F$. The symbol $R[S]$
means the free $R$\nbd-module with basis~$S$.

It is not difficult to see that $\mathcal{O}(k)$ is an object
of~$\qco(P)$; the $A_F$\nbd-module structure of~$\mathcal{O}(k)^F$ is
induced by the translation action of the monoid $T_F \cap \bZ^n$ on
the set of integral points in~$kF+T_F$. In fact, $\mathcal{O}(k)^F$ is
a free $A_F$\nbd-module of rank~$1$.

\begin{proposition}
\label{prop:cohomology}
  For $k \in \bZ$ and $j \in \bN$ there is an isomorphism
  \[H_j \check \Gamma (\mathcal{O} (k)) \iso
  \begin{cases}
    R[kP \cap \bZ^n] & \text{if } j=n \text{ and }
    k \geq 0 \ , \\
    R[({\mathrm{int}}\,kP) \cap \bZ^n] & \text{if }j=0 \text{ and } k <
    0\ , \\
    0 & \text{otherwise.}
  \end{cases}\]
\end{proposition}

\begin{proof}
  For $R = \bC$ this is the standard calculation of cohomology of the
  line bundles $\mathcal{O}(k)$ on the toric variety~$X_P$. Details
  are contained, for example, in \cite[2.5.3]{H-nonlin}; the
  calculation given there remains valid for arbitrary rings with unit.
\end{proof}

It follows from Theorem~\ref{thm:Ehrhart} that the number~$n_P$
defined in Corollary~\ref{cor:zeros_of_Ehrhart} can be characterised
as being minimal among those $k \geq 0$ for which $\mathcal{O}(-k-1)$
is not acyclic (\ie, $\check \Gamma \mathcal{O}(-k-1)$~is not
acyclic).

\subsection{Twisting sheaves}

\begin{definition}
  Let $Y \in \pre(P)$ and $k \in \bZ$. We define the $k$th {\it
    twist\/} of~$Y$, denoted $Y(k)$, as the objectwise tensor product
  of $\mathcal{O}(k)$ and~$Y$. Explicitly,
  \[Y(k) \colon F \mapsto Y(k)^F = \mathcal{O}(k)^F \tensor_{A_F} Y^F
  \ ,\] with structure maps induced by those of~$\mathcal{O}(k)$
  and~$Y$.
\end{definition}

By definition $Y(k)^F$ is isomorphic to the tensor product of~$Y^F$
with a free $A_F$\nbd-module of rank~$1$, hence $Y(k)^F$ is
non-canonically isomorphic to~$Y^F$. The following properties are
easily verified:

\begin{lemma}
  \begin{enumerate}
  \item For $Y \in \pre(P)$ and $k, \ell \in \bZ$ there is an
    isomorphism $Y(k)(\ell) \iso Y(k+\ell)$. Moreover $Y(0) \iso Y$,
    this last isomorphism being natural in~$Y$.
  \item If $Y \in \qco(P)$ then $Y(k) \in \qco(P)$ for every $k \in
    \bZ$.
  \item If $Y \in \hco(P)$ then $Y(k) \in \hco(P)$ for every $k \in
    \bZ$.
  \item If $Y \in \perf(P)$ then $Y(k) \in \perf(P)$ for every $k \in
    \bZ$.\qed
  \end{enumerate}
\end{lemma}

\subsection{Quasi-coherent functors}

\begin{definition}
  A {\it quasi-coherent functor~$Y$\/} is an object $Y \in \qco(P)$
  which is concentrated in chain degree~$0$.
\end{definition}

\begin{lemma}
  \label{lem:sections_extend}
  Let $Y$ be a quasi-coherent functor, let $F \in F(P)_0$, and let
  $s_F \in Y^F$. Then there exists $k \in \bZ$ and a map $f \colon
  \mathcal{O}(k) \rTo Y$ such that $s_F$~is in the image of the
  $F$\nbd-component of~$f$.
\end{lemma}

\begin{proof}
  The proof is a translation of the corresponding algebro-geometric
  fact \cite[II.5.14~(b)]{Ha-AG} into combinatorial language. To begin
  with, {\it we may assume that each face of~$P$ has a lattice point
    in its relative interior}. Indeed, we can replace~$P$ by its
  dilate $(n+1)P$; note that this does not change the poset~$F(P)_0$,
  nor does it change the cones, monoids and algebras constructed
  from~$P$. Twisting translates easily: the sheaf $\mathcal{O}(1)$
  computed with respect to~$(n+1)P$ is precisely the
  sheaf~$\mathcal{O}(n+1)$ computed with respect to~$P$.

  All algebras and modules constructed from~$P$ are $R$\nbd-submodules
  of the free $R$\nbd-module $R[\bZ^n]$. An element $v_F \in \bZ^n$
  gives rise to an element $x_F = 1 \cdot v_F \in R[\bZ^n]$. We think
  of the $x$\nbd-symbols as multiplicative, that is, we
  write~$x_F/x_G$ for the module element associated to the vector $v_F
  - v_G$.

  Now choose, for each $F \in F(P)_0$, a lattice point $v_F$ in the
  relative interior of~$F$. It is not difficult to see that
  $\mathcal{O}(N)^G = R[(NG+T_G) \cap \bZ^n]$, considered as a free
  $R$\nbd-module, has basis given by $Nv_G + S_G \subseteq \bZ^n$ so
  that
  \[\mathcal{O}(N)^G =  x_G^N A_G\ , \ N \in \bZ \ .\] Moreover, if $p
  \in \mathcal{O}(N)^G$ then $x_F^Mp \in \mathcal{O}(M+N)^G$ for all
  $M \geq 0$. In fact, $x_F^Mp$ is the image of~$x_F^M \tensor p$
  under the isomorphism $\mathcal{O}(M) \tensor \mathcal{O}(N) \iso
  \mathcal{O}(M+N)$ restricted to $G$\nbd-components.

  As notational convention, given an element $x \in Z^G$ for a
  coherent functor~$Z$ we call the image~$x|_H$ of~$x$ under the
  structure map $Z^G \rTo Z^H$ the restriction of~$x$ to~$H$.

  Let us now start with the actual proof. By definition of tangent
  cones $v_F - v_G \in T_G$ for each $G \in F(P)_0$ so that $x_F/x_G
  \in A_G$. One can show the following:
  \begin{assertion}
    \label{ass:algebra_localisation}
    For each $G \in F(P)_0$ we have
    \[T_{F \vee G} = T_G + \bR (v_F-v_G) \qquad \text{and} \qquad S_{F
      \vee G} = S_G + \bZ (v_F-v_G )\] where $F \vee G$ is the join
    of~$F$ and~$G$, that is, the smallest face of~$P$ containing $F
    \cup G$. Consequently, the algebra $A_{F \vee G}$ is obtained
    from~$A_G$ by localising by the single element~$x_F/x_G$ in the
    centre~$Z(A_G)$ of~$A_G$.
  \end{assertion}
  \noindent This implies that for a large enough positive integer $N$
  the element $s_G = (x_F/x_G)^N \cdot s_F|_{F \vee G}$ is in~$Y^{F
    \vee G} \iso A_{F \vee G} \tensor_{A_G} Y^G$, where $s_F$~is the
  element given in the formulation of the Lemma. We may pick an
  integer~$N$ which works for all $G \in F(P)_0$ simultaneously. Then
  for all~$G$, $x_G^N \tensor s_G$ is an element of $Y(N)^G =
  \big(\mathcal{O}(N) \tensor Y \big)^G$ which restricts to $x_F^N
  \tensor s_F|_{F \vee G} \in \big( \mathcal{O}(N) \tensor Y \big)^{F
    \vee G}$. (Note here that $(x_F^N \tensor s_F)|_{F \vee G} = x_F^N
  \tensor (s_F|_{F \vee G})$.)

  Now let $G \subseteq H \in F(P)_0$ be an arbitrary pair of non-empty
  faces of~$P$. We do not know whether the elements $x_G^N \tensor s_G
  |_H$ and~$x_H^N \tensor s_H$ agree, but we know that after
  restricting further to~$F \vee H$ both agree with $x_F^N \tensor
  s_F|_{F \vee H}$. Consequently, using~\ref{ass:algebra_localisation}
  again, for all large integers~$M$ we have equality
  \[\Big( \frac{x_F}{x_H} \Big)^M \cdot  x_G^N \tensor s_G|_H =
  \Big( \frac{x_F}{x_H} \Big)^M \cdot x_H^N \tensor s_H \in Y(N)^H \
  . \] Now multiplication with~$x_H^M$ yields an isomorphism $Y(N)^H
  \iso Y(M+N)^H$ so that the above equality becomes
  \[(x_F^M x_G^N) \tensor s_G|_H = (x_F^M x_H^N) \tensor s_H \in
  Y(M+N)^H \ , \] with both elements restricting to $x_F^{M+N} \tensor
  s_F|_{F \vee H}$ on~$F \vee H$.

  To sum up, we have shown that the family of elements $(x_F^M x_G^N)
  \tensor s_G$, $G \in F(P)_0$, determines an element~$e$ in
  $\lim\limits_\leftarrow Y(M+N)$, and hence an $R$\nbd-module
  homomorphism $R \rTo \lim\limits_\leftarrow Y(M+N)$ which sends $1
  \in R$ to~$e$. But then, by forcing equivariance, there is a map
  $\mathcal{O}(0) \rTo Y(M+N)$ which sends $1 \in A_G$ to
  $(x_F^Mx_G^N) \tensor s_G$. Twisting by $-(M+N)$ yields a map
  $\mathcal{O}(-M-N) \rTo Y$ such that $s_G$~is in the image of the
  $G$\nbd-component. This applies in particular to $G = F$ which is
  the case of the Lemma.
\end{proof}

\begin{corollary}
  \label{cor:generated_global}
  Let $Y$~be a quasi-coherent functor such that for all $F \in F(P)_0$
  the $A_F$\nbd-module~$Y^F$ is finitely generated. Then there are
  finitely many numbers $n_i \in \bZ$ and a map
  \[\bigoplus_i \mathcal{O}(n_i) \rTo Y\] which is surjective on each
  component.
\end{corollary}

\begin{proof}
  For $F \in F(P)_0$ choose generators $s^F_1,\, \cdots,\,
  s^F_{\ell(F)}$. By Lemma~\ref{lem:sections_extend} there are maps
  $f^F_i \colon \mathcal{O}(n^F_i) \rTo Y$ such that $s^F_i$~is in the
  image of~$f^F_i$. The required map is given by the sum
  \[\bigoplus_{F \in F(P)_0} \bigoplus_{i=1}^{\ell(F)}
  \mathcal{O}(n^F_i) \rTo Y \ .\]
\end{proof}

\begin{lemma}
  \label{lem:H_is_finite}
  Let $R$~be a left \textsc{noether}ian ring. Let $Y$~be a
  quasi-coherent functor such that for all $F \in F(P)_0$ the
  $A_F$\nbd-module~$Y^F$ is finitely generated. Then $H_k \check
  \Gamma (Y)$ is trivial for $k<0$ and $k>n$, and is a finitely
  generated $R$\nbd-module for all $k \in \bZ$.
\end{lemma}

\begin{proof}
  This follows the pattern of \cite[III.5.2]{Ha-AG}. Triviality of
  $H_k \check \Gamma(Y)$ for $k<0$ and $k>n$ is immediate as $\check
  \Gamma(Y)$ is concentrated in degrees~$0$ to~$n$, by construction.

  The Lemma is true for a finite sum of quasi-coherent functors of the
  form~$\mathcal{O}(k)$, by the calculation in
  Proposition~\ref{prop:cohomology}.  By
  Corollary~\ref{cor:generated_global} we can find a surjection $Z
  \rTo Y$ with $Z$ a finite sum of $\mathcal{O}(k)$s. Let $K$~denote
  the kernel; this is a quasi-coherent functor as well. By
  construction we obtain a short exact sequence of chain complexes
  \[0 \rTo \check \Gamma K \rTo \check \Gamma Z \rTo \check \Gamma Y
  \rTo 0\ .\] Now use increasing induction on~$k$ on the corresponding
  exact sequence snippet
  \[H_{k+1} \check \Gamma Z \rTo H_{k+1} \check \Gamma Y \rTo H_k \check
  \Gamma K \ ,\] starting with $k = -1$; by choice of~$Z$ the module
  on the left is finitely generated, and by what has been established
  by induction the module on the right is finitely generated as
  well. Since $R$~is left \textsc{noether}ian it follows that the
  middle module is finitely generated.
\end{proof}

\begin{proposition}
  \label{prop:qc_has_fg_cohomology}
  Suppose $R$ is left \textsc{noether}ian. Let $Y \in \hco(P)$ be such
  that $Y^F_k$ is finitely generated as an $A^F$\nbd-module for all $F
  \in F(P)_0$ and all $k \in \bZ$. Then $\check \Gamma (Y)$ is a
  (possibly unbounded) chain complex with finitely generated homology
  modules.
\end{proposition}

\begin{proof}
  Since $Y$~is a homotopy sheaf $H_q(Y) \colon F \mapsto H_q(Y^F)$ is
  a quasi-coherent functor (this uses Lemma~\ref{lem:F_G_loc} and the
  fact that taking homology is compatible with localisation). Since
  $R$~is left \textsc{noether}ian, and since all the modules $Y^F_k$
  are finitely generated, the modules $H_q(Y^F)$~are finitely
  generated as well. It follows from Lemma~\ref{lem:H_is_finite},
  applied to~$H_q(Y)$, and from
  Remark~\ref{rem:on_check_Gamma}~(\ref{item:limits}) that all the
  entries of the $E^2$\nbd-term of the spectral
  sequence~\ref{rem:other_ss} are finitely generated
  $R$\nbd-modules. Since this spectral sequence is concentrated in
  columns~$0$ to~$n$, its abutment $H_{p+q} \check \Gamma Y$ consists
  of finitely generated $R$\nbd-modules as well (making use the fact
  that $R$~is left \textsc{noether}ian once again).
\end{proof}

\subsection{Finiteness of the \textsc{\v Cech\/} complex}

We are now going to prove the following fundamental finiteness result:

\begin{theorem}
  \label{thm:Cech_finite}
  Let $R$~be a unital ring. Suppose that $R$~is commutative, or else
  left \textsc{noether}ian. Let $Y \in \perf(P)$. Then $\check \Gamma
  (Y)$~is a perfect complex of $R$\nbd-modules.
\end{theorem}

\begin{proof}
  Suppose first that $R$~is left \textsc{noether}ian. Since $Y$~is a
  perfect complex, there exists a homotopy vector bundle $V \in
  \hvect(P)$ together with a quasi-isomorphism $V \rTo Y$ by
  Lemma~\ref{lem:perf_qis_hvect}. Then the induced map $\check \Gamma
  V \rTo \check \Gamma Y$ is a quasi-isomorphism by
  Proposition~\ref{prop:cech_invariant}. Since $V$~is bounded so
  is~$\check \Gamma V$. Since $V^F$ consists of projective
  $A_F$\nbd-modules and since $A_F$~is free as an $R$\nbd-module,
  $\check \Gamma V$ consists of projective $R$\nbd-modules, and has
  finitely generated homology modules by
  Proposition~\ref{prop:qc_has_fg_cohomology}. But this means that
  $\check \Gamma V$ is chain homotopy equivalent to a strict perfect
  complex of $R$\nbd-modules \cite[1.7.13]{Rosen-K}.

  Now suppose that $R$~is commutative, but not
  \textsc{noether}ian. Then $Y$ can be replaced, up to
  quasi-isomorphism, by a bounded complex~$V$ in $\hvect(P) \cap
  \qco(P)$; this is true since the toric scheme~$X_P$ is projective
  over $\mathrm{Spec}\, R$ and thus has an ample line bundle, so we
  can appeal to~\cite[2.3.1~(d)]{TT-K}. Then $\check \Gamma (Y) \simeq
  \check \Gamma (V)$, and it is enough to prove the Theorem for~$V$
  only.  

  The complex~$V$ descends to a \textsc{noether}ian
  subring~$R_0$. More precisely, write $\qco(P)_0$ for the category
  $\qco(P)$ defined over a subring~$R_0$ instead of~$R$, and similarly
  for $\hvect(P)_0$. Then by \textsc{noether}ian approximation
  \cite[Appendix~C]{TT-K}, there is a \textsc{noether}ian
  subring~$R_0$ of~$R$, and a bounded complex $V_0 \in \hvect(P)_0
  \cap \qco(P)_0$ such that $V = R \tensor_{R_0} V_0$. By the first
  part of the proof there is a strict perfect complex~$B_0$ of
  $R_0$\nbd-modules which is chain homotopy equivalent to~$\check
  \Gamma V_0$. But then $\check \Gamma V \iso R \tensor_{R_0} \check
  \Gamma V_0$ is homotopy equivalent to $R \tensor_{R_0} B_0$, and the
  latter is strict perfect.
\end{proof}

\subsection{Canonical sheaves and suspension of chain complexes}

\begin{definition}
  \label{def:can_const_susp}
  \begin{enumerate}
  \item For $k \in \bZ$ and $C \in \Ch(R)$ we define $\mathcal{O}(k)
    \tensor C$ to be the sheaf given by
    \[\big(\mathcal{O}(k) \tensor C\big)^F = \mathcal{O}(k)^F \tensor_R
    C \quad \text{for } F \in F(P)_0\] with structure maps induced by
    those of~$\mathcal{O}(k)$. We call $\mathcal{O}(k) \tensor C$ the
    $k$th~{\it canonical sheaf associated to~$C$}.
  \item Let $C$ be an $R$\nbd-module chain complex.  We denote by
    $\con(C)$ the constant $F(P)_0$-diagram with value~$C$ and
    identity structure maps.
  \item \label{item:susp} The $n$th {\it suspension\/} $C[n]$ of the
    chain complex $C \in \Ch(R)$ is defined by $C[n]_k = C_{k-n}$, and
    multiplying the differentials with the sign~$(-1)^n$.
  \end{enumerate}
\end{definition}

\begin{lemma}
  \label{lem:Gamma_of_psi_k}
  Let $C \in \Ch(R)$. If $k<0$ is an integer such that $E_P(k)=0$
  (\ie, such that $kP$ has no lattice points in its interior), then
  $\check \Gamma \big( \mathcal{O}(k) \tensor C \big)$ is acyclic.
\end{lemma}

\begin{proof}
  Let $D_{*,*} = D_{*,*} \big( \mathcal{O}(k) \tensor C\big)$ be the
  double chain complex associated to~$ \mathcal{O}(k) \tensor C$,
  cf.~(\ref{eq:associated_double_complex}). The $A_F$\nbd-module
  $\mathcal{O}(k)^F \iso A_F$ is a free $R$\nbd-module for each $F \in
  F(P)_0$, and the \textsc{\v Cech\/} complex $\check \Gamma
  \mathcal{O}(k)$ has the property that the image of each chain module
  under the differential is a free $R$\nbd-module (since all maps are
  given by inclusion of bases). Hence by the \textsc{K\"unneth\/}
  formula \cite[3.6.1]{Weibel-intro} the homology of the horizontal
  chain complex $D_{*,p}$ fits into a short exact sequence
  \[H_k \big( \check \Gamma \mathcal{O}(k) \big) \tensor_R C_p \rTo
  H_k \big( \check \Gamma \mathcal{O}(k) \tensor_R C_p \big) \rTo
  \mathrm{Tor}^R_1 \big( H_{n-1} \check \Gamma \mathcal{O}(k),\,
  C_p\big) \ .\] By hypothesis $E_P(k)=0$ so that $\check \Gamma
  \mathcal{O}(k)$ is acyclic by Proposition~\ref{prop:cohomology}. So
  first and third term of the short exact sequence are trivial, hence
  so is the middle term.

  But this means that the spectral sequence~(\ref{eq:ss_2})
  associated to~$D_{*,*}$ has trivial $E^1$\nbd-term, hence its
  abutment $H_* \check \Gamma \big(  \mathcal{O}(k) \tensor C \big)$
  is trivial too. This proves the Lemma.
\end{proof}

\begin{lemma}
  \label{lem:susp_limit}
  For every chain complex $C \in \Ch(R)$ we have a canonical
  quasi-isomorphism $C[n] \rTo^\simeq \check \Gamma \con(C)$.
\end{lemma}

\begin{proof}
  Let $A_{*,*}$ denote the complex~$C$ considered as a double chain
  complex concentrated in column~$n$. That is, we're looking at the
  double chain complex with $A_{n,k} = C_k$, $A_{j,k}=0$ for $j \ne
  n$, and vertical differential the differential of~$C$ multiplied
  with the sign~$(-1)^n$. Then the total complex of~$A_{*,*}$ is
  precisely~$C[n]$.

  The double chain complex $A_{*,*}$ maps into the double chain
  complex $D_{*,*} = D_{*,*} \big( \con(C) \big)$ associated
  to~$\con(C)$ by the diagonal map~$\Delta$ given by
  \[C_t \rTo \bigoplus_{\dim F = n} C_t = D_{n,t}\ , \ t \in \bZ \ .\]
  This defines indeed a map of double chain complexes by the
  properties of incidence numbers; more precisely, the horizontal
  chain complex $D_{*,t}$ is the tensor product of~$C_t$ with the dual
  of the chain complex computing the integral cellular homology
  of~$P$, and the map of horizontal chain complexes $A_{*, t} \rTo
  D_{*,t}$ is the tensor product of~$C_t$ with the dual of the
  augmentation map. In particular, the map is a quasi-isomorphism with
  respect to ``horizontal'' homology. By the spectral sequence
  argument~\ref{prop:ss_argument}~(2), $\Delta$~induces a
  quasi-isomorphism
  \[C[n] = \mathrm{Tot}\, A_{*,*} \rTo[l>=3em]^{\mathrm{Tot}\, \Delta}
  \mathrm{Tot}\, D_{*,*} = \check \Gamma \big(\con(C) \big) \
  . \qedhere\]
\end{proof}

\begin{lemma}
  \label{lema:con_psi0}
  For every chain complex $C \in \Ch(R)$ we have a canonical
  map $\con(C) \rTo \mathcal{O}(0) \tensor C$, induced by the
  inclusions of $F$-components
  \[\con(C)^F = C \iso R[\{0\}] \tensor_R C \rTo R[\bZ^n \cap C_F]
  \tensor_R C = \big(\mathcal{O}(0) \tensor C\big)^F \ ,\] which in
  turn induces a quasi-isomorphism $\check \Gamma \con(C) \rTo^\simeq
  \check \Gamma \big(\mathcal{O}(0) \tensor C\big)$.
\end{lemma}

\begin{proof}
  First recall that $\mathcal{O}(0)^F = A_F = R[S_F]$ for each $F \in
  F(P)_0$; this means that we can consider $\mathcal{O}(0)^F$ as a
  $\bZ^n$\nbd-graded $R$\nbd-module with homogeneous components~$0$
  or~$R$. The double chain complex $D_{*,*} \big( \mathcal{O}(0)
  \tensor C \big)$ associated to $\mathcal{O}(0) \tensor C$ is a
  double chain complex of $\bZ^n$\nbd-graded $R$\nbd-modules with
  horizontal and vertical differentials respecting the grading. We can
  thus concentrate on homogeneous components one at a time.  As
  explained in \cite[2.5.3]{H-nonlin}, the component of degree $0 \in
  \bZ^n$ of the horizontal chain complex $D_{*,t} \big( \mathcal{O}(0)
  \tensor C \big)$ is the tensor product of~$C_t$ with the dual of the
  chain complex calculating the integral cellular homology of~$P$ so
  that its horizontal homology is concentrated in column~$n$ and has
  value~$C_t$. For all homogeneous degrees different from~$0$ the
  chain complex is acyclic ({\it loc.\,cit.}).

  On the other hand, we have a canonical isomorphism of diagrams
  \[\con (C) \iso \con(R) \tensor C\] where $\con(R)$ is the
  constant diagram with value~$R$, considered as a diagram of chain
  complexes concentrated in chain degree~$0$. We can think of~$R$ as a
  $\bZ^n$\nbd-graded module concentrated in degree $0 \in \bZ^n$. Thus
  the associated double chain complex $D_{*,*} \big( \con(R) \tensor C
  \big)$ consists of $\bZ^n$\nbd-graded modules with differentials
  preserving the grading. In homogeneous degree~$0$, the horizontal
  chain complexes $D_{*,t} \big( \con(R) \tensor C \big)$ and $D_{*,t}
  \big( \mathcal{O}(0) \tensor C \big)$ agree, in non-zero degrees the
  horizontal chain complex $D_{*,t} \big( \con(R) \tensor C \big)$ is
  the zero-complex.

  We have an obvious map of double complexes
  \[\omega \colon D_{*,*} \big( \con(R) \tensor C \big) \rTo D_{*,*} \big(
  \mathcal{O}(0) \tensor C \big) \ ,\] the inclusion of degree~$0$
  components, which by the previous two paragraphs induces an
  isomorphism on horizontal homology modules. By the spectral sequence
  argument~\ref{prop:ss_argument}~(2) the composite
  \[
  \begin{split}
    \check \Gamma \con(C) \iso & \check \Gamma \big( \con(R) \tensor C
    \big) = \mathrm{Tot}\, D_{*,*} \big( \con(R) \tensor C \big)\\
    &\qquad\qquad \rTo[l>=3em]^\simeq_\omega \mathrm{Tot}\, D_{*,*}
    \big( \mathcal{O}(0) \tensor C \big) = \check \Gamma \big(
    \mathcal{O}(0) \tensor C \big)
  \end{split}
  \] is thus a quasi-isomorphism.
\end{proof}

\subsection{A model for suspension}

For any presheaf $Y \in \pre(P)$ we let $Y[n]$ denote the $n$th
suspension of~$Y$, that is, the diagram given by $F \mapsto Y[n]^F =
Y^F[n]$, cf.~Definition \ref{def:can_const_susp}~\eqref{item:susp},
with structure maps induced by those of~$Y$.

\begin{definition}
  Let $F$~be a face of~$P$ (possibly empty), and let $Y \in
  \pre(P)$. We define a new presheaf $F_*Y$ by the rule
  \[F_*Y \colon G \mapsto Y^{G \vee F}\] where $G \vee F$ is the join
  of~$G$ and~$F$, that is, the smallest face of~$P$ containing $G \cup
  F$, and $(F_*Y)^G = Y^{F \vee G}$ is considered as an
  $A_G$\nbd-module chain complex by restriction of scalars.
\end{definition}

\begin{remark}
  If $R$~is a commutative ring and $Y \in \qco(P)$~is a sheaf, let
  $\mathcal{F}$~denote the chain complex of quasi-coherent sheaves
  on~$X_P$ determined by~$Y$. Then $F_*Y$ corresponds to $j_*
  (\mathcal{F}|_{U_F}) = j_* \big(\widetilde {Y^F}\big)$ where $j
  \colon U_F \rTo X_P$ is the inclusion.
\end{remark}

The construction of~$F_*Y$ is natural in~$Y$: For $F \subseteq
F^\prime$ a pair of faces of~$P$, the structure maps of~$Y$ induce a
map of presheaves $F_*Y \rTo F^\prime_*Y$. Moreover, every entry
of~$F_*(Y)$ is an $A_F$-module, by restriction of scalars
($A_\emptyset = R$ here) so that $\check \Gamma F_*Y$ is a chain
complex of $A_F$\nbd-modules. Hence the following definition is
meaningful:

\begin{definition}
  For $Y \in \pre(P)$ we define a new presheaf $\sigma Y$ by
  \[F \mapsto (\sigma Y)^F = \check \Gamma (F_*Y) \ .\]
\end{definition}

\begin{lemma}
  \label{lemma:susp_is_sigma}
  Let $Y \in \pre(P)$. There is a natural map of presheaves
  \[\alpha \colon Y[n] \rTo \sigma Y\] which is a quasi-isomorphism on
  each component. In particular, if~$Y$is a homotopy sheaf then so
  is~$\sigma Y$, and if~$Y$ is a perfect complex so is~$\sigma Y$.
\end{lemma}

\begin{proof}
  Let $\gamma Y$ denote the presheaf $(\gamma Y)^F = \check \Gamma
  \con(Y^F)$. For any face $F \in F(P)_0$ the structure maps of~$Y$
  induce a map of diagrams $\con(Y^F) \rTo F_*Y$ and thus a map
  $\check \Gamma \con(Y^F) \rTo \check \Gamma F_*Y$. This construction
  is natural in~$F$ so we obtain maps of presheaves
  \[Y[n] \rTo \gamma Y \rTo \sigma Y \ ,\] the first one consisting of
  the canonical quasi-isomorphisms of Lemma~\ref{lem:susp_limit}. The
  composition of these two maps is the~$\alpha$ of the Lemma, and we
  are left to prove that the map
  \[(\gamma Y)^F = \check \Gamma \con(Y^F) \rTo \check \Gamma F_*Y =
  (\sigma Y)^F\] is a quasi-isomorphism for each $F \in F(P)_0$; we
  will use the (by now familiar) spectral sequence comparison
  argument.

  Write $\mathrm{st}\,F = \{G \in F(P)_0 \,|\, G \supseteq F\}$, a
  sub-poset of~$F(P)_0$. Given a diagram $A \colon F(P)_0 \rTo \RMod$
  we can consider its restriction $A|_{\mathrm{st}\,F}$ to the poset
  $\mathrm{st}\,F$. Conversely, a diagram $B \colon \mathrm{st}\, F
  \rTo \RMod$ can be extended to a diagram
  \[F_* B \colon F(P)_0 \rTo \RMod\ , \quad G \mapsto B^{G \vee F} \
  .\] In fact, extension and restriction form an adjoint pair, with
  restriction being the left adjoint. Both functors are exact.

  Now let $B \colon \mathrm{st}\,F \rTo \RMod$ be given, and let $B
  \rTo I^\bullet$ be an injective resolution of~$B$. Then $F_* B \rTo
  F_* I^\bullet$ is an injective resolution of~$F_* B$. Consequently,
  we have
  \begin{align*}
    \lim_\leftarrow{}^q B &= H^q \lim_{\mathrm{st}\, F} I^\bullet \\
    &= H^q \hom \big((\con R)|_{\mathrm{st}\,F},\, I^\bullet\big) \\
    &= H^q \hom (\con R,\, F_* I^\bullet) \\
    &= H^q \lim F_* I^\bullet \\
    &= \lim_\leftarrow{}^q F_* B \ .
  \end{align*}
  On the other hand, $\mathrm{st}\,F$ has minimal element~$F$ so that
  $\lim B = B^F$, and $\lim$ is exact; that is,
  \[\lim_\leftarrow{}^q B =
  \begin{cases}
    B^F & \text{ if } q = 0\\ 0 & \text{ else.}
  \end{cases}
  \]
  These calculations apply in particular to $B = A|_{\mathrm{st}\, F}$
  for $A$~an $F(P)_0$\nbd-diagram of $R$\nbd-modules. Since
  $\big(F_*(A|_{\mathrm{st}\,F})\big)^F = A^F$ this means that
  \[\lim_\leftarrow{}^q F_*(A|_{\mathrm{st}\,F}) =
  \begin{cases}
    A^F & \text{ if } q = 0\\ 0 & \text{ else.}
  \end{cases}
  \]
  The calculations also imply that the obvious map
  $\con(B^F)|_{\mathrm{st}\,F} \rTo B$, for arbitrary~$B$ as before
  and its adjoint $\con B^F \rTo F_*B$ induce isomorphisms
  \[\lim_\leftarrow{}^q \big(\con B^F\big)|_{\mathrm{st}\,F} \rTo^\iso
  \lim_\leftarrow{}^q B \quad\text{and}\quad \lim_\leftarrow{}^q \con
  B^F \rTo^\iso \lim_\leftarrow{}^q F_*B\] as both are the identity on
  $F$-components.

  Let us return to the map
  \[(\gamma Y)^F = \check \Gamma \con(Y^F) \rTo \check \Gamma F_*Y =
  (\sigma Y)^F \ .\] It is induced by a map of double chain complexes
  \[D_{*,*} (\con Y^F) \rTo D_{*,*} (F_* Y)\] which, when restricted
  to $t$th horizontal chain complexes, is a map of chain complexes
  computing the higher derived limits $\lim\limits_\leftarrow{}^{n-*}$
  of the diagrams $A = \con Y^F_t$ and~$F_*
  \big(Y_t|_{\mathrm{st}\,F}\big)$, respectively, by
  Remark~\ref{rem:on_check_Gamma}~\eqref{item:limits}; by the
  calculation above, applied to~$A$ and~$B = Y_t|_{\mathrm{st}\,F}$,
  this yields an isomorphism on homology modules for all values
  of~$*$. In other words, the $E^1$\nbd-terms of the two spectral
  sequences are isomorphic, hence their abutments are as well, so the
  map in question is a quasi-isomorphism as claimed.
\end{proof}

\subsection{Relating $\mathcal{O}(0)\tensor \check \Gamma$ to $n$th
  suspension}
We proceed to construct a map between the functors $\mathcal{O}(0)
\tensor \check \Gamma$ and $n$th suspension which will be used later
in $K$\nbd-theoretical computations.

\begin{numpar}
  Let $Y \in \pre(P)$, and fix $F \in F(P)_0$. The
  structure maps of~$Y$ induce a canonical map $Y^F \rTo \lim F_*Y$ and
  thus, by forcing equivariance, a map
  \begin{equation}
    \label{eq:map_rho_F}
    \rho^F \colon \mathcal{O}(0) \tensor Y^F \rTo F_*Y \ .
  \end{equation}
\end{numpar}

\begin{lemma}
  \label{lem:gamma_rho_F_qiso}
  For $Y \in \hco(P)$ the map $\check \Gamma (\rho^F) \colon \check
  \Gamma \big (\mathcal{O}(0) \tensor Y^F \big) \rTo \check \Gamma
  F_*Y$ is a quasi-isomorphism.
\end{lemma}

\begin{proof}
  The map $\check \Gamma (\rho^F)$ fits into a commutative square diagram
  \begin{diagram}
    \check \Gamma \big (\mathcal{O}(0) \tensor Y^F \big)
    & \rTo[l>=4em]^{\check \Gamma (\rho^F)} & \check \Gamma F_* Y \\
    \uTo && \uTo \\
    \check \Gamma \mathrm{con}\, (Y^F) & \lTo & Y^F[n]
  \end{diagram}
  Left, bottom and right map are quasi-isomorphisms by
  Lemmas~\ref{lema:con_psi0}, \ref{lem:susp_limit}
  and~\ref{lemma:susp_is_sigma}, respectively, hence the top map is a
  quasi-isomorphism as well.
\end{proof}

\begin{numpar}
  Let $Y \in \pre(P)$. The structure maps of~$Y$ induce, for each $F
  \in F(P)_0$, a map of presheaves $Y \rTo F_*Y$ and hence a map of
  chain complexes $\check \Gamma Y \rTo \check \Gamma F_*Y = (\sigma
  Y)^F$. Since this is natural in~$F$ we obtain a map
  \[\check \Gamma Y \rTo \lim_{F \in F(P)_0} \check \Gamma F_*Y =
  \lim_{F \in F(P)_0} (\sigma Y)^F\] which, by forcing equivariance,
  defines a map
  \begin{equation}
    \label{eq:map_beta} \beta \colon \mathcal{O}(0) \tensor \check
    \Gamma Y \rTo \sigma Y \ .
  \end{equation}
\end{numpar}

\begin{lemma}
  \label{lem:gamma_psi0_gamma}
  Let $Y \in \hco(P)$. Then $\check \Gamma(\beta) \colon \check \Gamma
  \big( \mathcal{O}(0) \tensor \check \Gamma Y\big) \rTo \check \Gamma
  (\sigma Y)$ is a quasi-isomorphism of chain complexes of
  $R$\nbd-modules.
\end{lemma}

Before we delve into the proof, we fix conventions regarding triple
chain complexes. Suppose we have a threefold $R$\nbd-module chain
complex $A_{*,*,*}$: a $\bZ^3$\nbd-indexed collection of
$R$\nbd-modules~$A_{x,y,z}$ together with pairwise commuting
differentials
\begin{align*}
  \bar \partial_x \colon A_{x,y,z} \rTo A_{x-1,y,z} \\
  \bar \partial_y \colon A_{x,y,z} \rTo A_{x,y-1,z} \\
  \bar \partial_z \colon A_{x,y,z} \rTo A_{x,y,z-1}
\end{align*}
which square to trivial maps. That is, we are looking at an object of
the category $\Ch (\Ch (\Ch (R)))$. Then we can define new differentials by
\[\partial_x = \bar \partial_x \ , \quad \partial_y = (-1)^x
\bar \partial_y \ , \quad \partial_z = (-1)^{x+y} \bar \partial_z\]
which are easily checked to anti-commute: $\partial_i \partial_j =
(\delta_{i,j} -1) \partial_j \partial_i$ (where $\delta_{i,j}$ is the
usual \textsc{Kronecker\/} delta symbol). We say that the graded
module $A_{*,*,*}$ together with the maps~$\partial_i$, $i=x,y,z$ is a
triple chain complex (in analogy to the established usage of the term
double complex in the literature).  We define the total complex
$\mathrm{Tot}^{x,y,z} (A_{*,*,*})$ by setting
\[\mathrm{Tot}^{x,y,z} (A_{*,*,*})_n = \bigoplus_{x+y+z=n} A_{x,y,z}\]
equipped with the differential defined by $\partial = \partial_x
+ \partial_y + \partial_z$. This is an $R$\nbd-module chain complex,
the relation $\partial^2 = 0$ is easily verified.

The relevant observation here is that one can do the totalisation in
two steps. Define $\mathrm{Tot}^{y,z} (A_{*,*,*})$ by
\[\mathrm{Tot}^{y,z} (A_{*,*,*})_{p,q} = \bigoplus_{y+z = q}
A_{p,y,z} \ ;\] this is a double chain complex when equipped with
``horizontal'' differential $\partial^h = \partial_x$ and ``vertical''
differential $\partial^v = \partial_y + \partial_z$. It is a matter of
tracing definitions to see that we have an equality of chain complexes
\[\mathrm{Tot}\, \big( \mathrm{Tot}^{y,z} (A_{*,*,*}) \big) =
\mathrm{Tot}^{x,y,z} (A_{*,*,*}) \ .\]

Similarly, we can define a double chain complex  $\mathrm{Tot}^{x,z} (A_{*,*,*})$ by
\[\mathrm{Tot}^{x,z} (A_{*,*,*})_{p,q} = \bigoplus_{x+z = p}
A_{x,q,z} \ ;\] equipped with ``horizontal'' differential $\partial^h
= \partial_x + \partial_z$ and ``vertical'' differential $\partial^v
= \partial_y$. We then have an equality of chain complexes
\[\mathrm{Tot}\, \big( \mathrm{Tot}^{x,z} (A_{*,*,*}) \big) =
\mathrm{Tot}^{x,y,z} (A_{*,*,*}) \ .\]

\begin{proof}[Proof of Lemma~\ref{lem:gamma_psi0_gamma}]
  The map $\check \Gamma (\beta)$ is a map of $R$\nbd-module chain
  complexes which can, in fact, be described by a map of triple chain
  complexes. In more detail, define
  \[A_{x,y,z} = \bigoplus_{\dim G = n-x} \bigoplus_{\dim F = n-y}
  \mathcal{O}(0)^G \tensor_R Y^F_z \ ,\] equipped with differentials
  $\bar \partial_z$ given by the differential in~$Y^F$, and
  differentials $\bar \partial_x$ and $\bar \partial_y$ determined by
  incidence numbers $[G_1:G_2]$ and $[F_1:F_2]$,
  respectively. Similarly, define
  \[B_{x,y,z} = \bigoplus_{\dim G = n-x} \bigoplus_{\dim F = n-y} Y^{G
    \vee F}_z \ ,\] equipped with differentials as above. Then both
  $A_{x,y,z}$ and~$B_{x,y,z}$ are threefold chain complexes and thus
  determine, by modification of the differentials, triple chain
  complexes as explained above. The structure maps of~$Y$ induce a map
  of triple chain complexes $\gamma \colon A_{x,y,z} \rTo
  B_{x,y,z}$. A tedious but straightforward tracing of signs and direct
  sums involved shows:
  \begin{enumerate}[(i)]
  \item $D_{*,*}\big( \mathcal{O}(0) \tensor \check \Gamma Y \big) = \mathrm{Tot}^{y,z} (A_{*,*,*})$,
  \item $D_{*,*} \big( \sigma Y \big) = \mathrm{Tot}^{y,z} (B_{*,*,*})$,
  \item $D_{*,*} (\beta) = \mathrm{Tot}^{y,z} (\gamma)$.
  \end{enumerate}
  Since $\mathrm{Tot} \circ \mathrm{Tot}^{y,z} = \mathrm{Tot}^{x,y,z}$
  this implies that
  \begin{equation}
    \label{eq:Tot_is_gamma}
    \mathrm{Tot}^{x,y,z} (\gamma) = \check \Gamma (\beta) \ .
  \end{equation}

  Let us now consider the double chain complex map $\mathrm{Tot}^{x,z}
  (\gamma)$. We claim that {\it $\mathrm{Tot}^{x,z} (\gamma)$ induces
    a quasi-isomorphism on horizontal chain complexes}. By the usual
  spectral sequence argument~\ref{prop:ss_argument}, this implies that
  $\mathrm{Tot}^{x,y,z} (\gamma) = \mathrm{Tot}\big(\mathrm{Tot}^{x,z}
  (\gamma) \big)$ is a quasi-isomorphism of $R$\nbd-module chain
  complexes, hence so is $\check \Gamma(\beta)$ in view
  of~(\ref{eq:Tot_is_gamma}).

  Fix an index $q \in \bZ$. The $q$th row of the source of
  $\mathrm{Tot}^{x,z} (\gamma)$ has $p$th entry
  \[\bigoplus_{\dim F = n-q} \bigoplus_{x+z=p} \bigoplus_{\dim G =
    n-x} \mathcal{O}(0)^G \tensor_R Y_z^F\] and is thus of the
  form
  \[\bigoplus_{\dim F = n-q} \check \Gamma \big (\mathcal{O}(0)
  \tensor Y^F \big) \ ,\] up to the (constant!) sign $(-1)^q$ in the
  differential of~$Y^F$. In particular this is non-trivial only for $0
  \leq q \leq n$.

  The $q$th row of the target of $\mathrm{Tot}^{x,z} (\gamma)$ has
  $p$th entry
  \[ \bigoplus_{\dim F = n-q} \bigoplus_{x+z=p} \bigoplus_{\dim G =
    n-x} Y^{G \vee F}_z = \bigoplus_{\dim F = n-q} \bigoplus_{x+z=p}
  \bigoplus_{\dim G = n-x} (F_*Y)^G_z\] and is thus of the form
  \[\bigoplus_{\dim F = n-q} \check \Gamma F_*Y \ ,\] up to the
  (constant!) sign $(-1)^q$ in the differential of~$F_*Y$. In particular this is non-trivial only for $0
  \leq q \leq n$.

  Thus in row~$q$ the map $\mathrm{Tot}^{x,z} (\gamma)$ is, up to
  sign~$(-1)^q$ in the differentials of source and target, the direct
  sum of the maps~$\rho^F$ defined in~(\ref{eq:map_rho_F}) with $\dim
  F = n-q$. By Lemma~\ref{lem:gamma_rho_F_qiso} this means that
  $\mathrm{Tot}^{x,z} (\gamma)$ is a quasi-isomorphism of horizontal
  chain complexes as claimed.
\end{proof}

\section{Splitting the $K$-theory}

Let $R$ be a ring with unit. For this entire section we assume that
$R$~is commutative, or else left \textsc{noether}ian.

\subsection{Reduced $K$-theory}
\label{sec:reduced-k}

Recall that an $R$\nbd-module chain complex~$C$ is called {\it
  perfect\/} if it is quasi-isomorphic to a strict perfect complex,
that is, a bounded complex~$B$ of finitely generated projective
$R$\nbd-mod\-ules; if this is the case, there will always be a
quasi-isomorphism $B \rTo C$.  Write $\perf(R)$ for the category of
perfect chain complexes of $R$\nbd-modules. We write $K(R)$ for the
$K$\nbd-theory space of the complicial bi\textsc{Waldhausen}
category $\perf(R)$ equipped with quasi-iso\-mor\-phisms as weak
equivalences, and the degreewise split monomorphisms with cokernel
in~$\perf(R)$ as cofibrations.

For an $n$\nbd-dimensional polytope $P \subset \bR^n$ we have defined
the category $\perf(P)$ of perfect complexes in~\ref{def:perfect};
recall that an object of~$\perf(P)$ is a diagram indexed by the face
lattice of~$P$ with values in perfect chain complexes of modules over
different rings, subject to a gluing condition. The $K$\nbd-theory
space of $\perf(P)$ is denoted by~$K(P)$,
cf.~\S\ref{subsec:algebraic-k-theory}.

Let $\perf(P)^{[0]}$ denote the full subcategory of those $Y \in
\perf(P)$ such that $\check \Gamma (Y)$ is acyclic,
cf.~Definition~\ref{def:cech_complex}~(\ref{item:cech3}). This is a
complicial bi\textsc{Wald\-hausen} category with the usual
conventions. Its associated $K$\nbd-theory space is called the {\it
  reduced $K$-theory of~$P$} and denoted~$\tilde K(P)$.

We call a map $f \colon Y \rTo Z$ in $\perf(P)$ an {\it
  $h_{[0]}$\nbd-equivalence\/} if $\check \Gamma(f)$ is a
quasi-isomorphism; with respect to these maps as weak equivalences,
$\perf(P)$ is (yet another) complicial bi\textsc{Waldhausen}
category. Note that every quasi-isomorphism in~$\perf(P)$ is an
$h_{[0]}$\nbd-equivalence as the functor $\check \Gamma$ preserves
quasi-isomorphisms by Proposition~\ref{prop:cech_invariant}.

We will need the functor
\[\psi_0 \colon \perf(R) \rTo \perf(P) \ , C \mapsto \mathcal{O}(0)
\tensor C \ .\] It is easy to see that $\psi_0$ takes values in
perfect complexes. Indeed, for $C \in \perf(C)$ there is a strict
perfect complex of $R$\nbd-modules~$D$ which is quasi-isomorphic
to~$C$. Since $A_F$~is a free $R$-module, for each $F \in F(P)_0$,
taking tensor product with~$A_F$ over~$R$ is exact. Consequently,
$\psi_0(Y)^F$ is quasi-isomorphic to~$A_F \tensor_R D$ which is a
strict perfect complex of $A_F$\nbd-modules.

\begin{proposition}
  \label{prop:reduced_ktheory}
  There is a fibration sequence of $K$-theory spaces
  \[\tilde K(P) \rTo K(P) \rTo^{\check \Gamma} K(R)\] which has a
  section up to homotopy and up to sign induced by the
  functor~$\psi_0$. Hence there is a splitting up to homotopy
  \[\tilde K(P) \times K(R) \simeq K(P) \ .\]
\end{proposition}

\begin{proof}
  By the Fibration Theorem \cite[1.6.4]{W-Ktheory} the sequences of
  exact functors of bi\-\textsc{Wald\-hausen} categories
  \begin{equation}
    \label{eq:fibration_0}
    (\perf(P)^{[0]},\, w) \rTo^{\subseteq} (\perf(P),\, w) \rTo
    (\perf(P),\,h_{[0]}) \ ,
  \end{equation}
  where $w$ stands for quasi-isomorphisms as weak equivalences,
  induces a fibration sequence of $K$\nbd-theory spaces.

  We have exact functors
  \[\perf(R) \rTo^{\psi_0} (\perf(P),\, h_{[0]}) \quad \text{and}
  \quad (\perf(P),\, h_{[0]}) \rTo^{\check \Gamma} \perf(R) \ ,\] the
  latter being well defined by Theorem~\ref{thm:Cech_finite}. By
  Lemmas~\ref{lem:susp_limit} and~\ref{lema:con_psi0} we have a
  natural weak equivalence of functors from the $n$th suspension $C
  \mapsto C[n]$ to the composition $\check \Gamma \circ \psi_0$. Since
  suspension induces a self homotopy equivalence on the $K$\nbd-theory
  space~$K(R)$, the functor $\check \Gamma$ is surjective on
  homotopy groups.

  By Lemmas~\ref{lemma:susp_is_sigma} and~\ref{lem:gamma_psi0_gamma}
  there is a chain of natural transformation of functors represented by
  \[Y[n] \rTo \sigma Y \lTo \mathcal{O}(0) \tensor \check \Gamma (Y)
  = \psi_0 \circ \check \Gamma (Y)\] which is
  in fact a chain of $h_{[0]}$\nbd-equivalences of functors. Thus
  $\check \Gamma$ is injective on homotopy groups.

  In total, we have shown that the functor $\check \Gamma$ induces a
  homotopy equivalence from the $K$-theory of the base of the fibration
  sequence~(\ref{eq:fibration_0}) to~$K(R)$. The resulting fibration
  sequence
  \[\tilde K(P) \rTo K(P) \rTo^{\check \Gamma} K(R)\] has section up
  to homotopy and up to sign induced by~$\psi_0$ (as the composition
  $\check \Gamma \circ \psi_0$ is weakly equivalent to $n$th
  suspension, just as argued above) which yields the desired
  splitting.
\end{proof}

\subsection{Further splitting}
\label{sec:further-splitting}

If $\check \Gamma \big( \mathcal{O}(-1)\big)$ happens to be acyclic, we can split
off a further copy of~$K(R)$ from the reduced $K$\nbd-theory~$\tilde
K(P)$; acyclicity of $\check \Gamma \mathcal{O}(-j)$ for $j>1$ allows
to iterate the procedure. The argument is virtually the same as in
\S\ref{sec:reduced-k}, but the functors involve additional twisting.
We record the details.

For a given integer $k \geq 0$, let $\perf(P)^{[k]}$ denote the full
subcategory of those $Y \in \perf(P)$ such that $\check \Gamma Y(j)$
is acyclic for $0 \leq j \leq k$. This is a complicial
bi\textsc{Wald\-hausen} category with the usual conventions. Its
associated $K$\nbd-theory space is denoted~$K(P, [k])$; in
particular, $K(P, [0]) = \tilde K(P)$.

We call a map $f \colon Y \rTo Z$ in $\perf(P)$ an {\it
  $h_{[k]}$\nbd-equivalence\/} if $\check \Gamma f(j)$ is a
quasi-isomorphism for $0 \leq j \leq k$; with respect to these maps as
weak equivalences, $\perf(P)$ is (yet another) complicial
bi\textsc{Waldhausen} category. Note that every quasi-isomorphism
in~$\perf(P)$ is an $h_{[k]}$\nbd-equivalence as both twisting and the
functor $\check \Gamma$ preserve quasi-isomorphisms.

We will need the functors
\[\psi_{k} \colon \perf(R) \rTo \perf(P) \ , C \mapsto \mathcal{O}(k)
\tensor C\] ($k \in \bZ$ here). One can show that this functor takes
indeed values in perfect complexes; the argument is as
for~$\psi_0$. As a matter of notation, let us also introduce the $k$th
twist functor
\[\theta_k \colon \perf(P) \rTo \perf(P) \ , Y \mapsto Y(k) \ .\]

Recall that the polytope~$P$ determines a polynomial $E_P(x)$ with
rational coefficients such that $|E_P(-j)|$ is the number of integral
points in the interior of~$-jP$ for integers $j \geq 1$,
cf.~Theorem~\ref{thm:Ehrhart}. It follows that {\it if $E_P(-j) = 0$
  for some $j > 1$, then $E_P(-\ell)=0$ for $0 < \ell \leq
  j$}. Let~$n_P$ be the number of distinct integral roots of~$E_P(x)$;
then $n_P \in [0,n]$, and if $n_P \neq 0$ then $n_P$~is maximal among
the negatives of integer roots of~$E_P(x)$.

\begin{proposition}
  \label{prop:split_again}
  For $1 \leq \ell \leq n_P$ there is a fibration sequence of
  $K$\nbd-theory spaces
  \[K(P, [\ell]) \rTo K(P, [\ell-1]) \rTo[l>=3em]^{\check \Gamma \circ
    \theta_\ell} K(R)\] which has a section up to homotopy and up to
  sign induced by the functor $\psi_{-\ell}$. Consequently, we have a
  homotopy equivalence
  \[K(P, [\ell]) \times K(R) \simeq K(P, [\ell-1]) \ .\]
\end{proposition}

\begin{proof}
  The sequence of bi\textsc{Waldhausen\/} categories
  \[(\perf(P)^{[\ell]},\,w) \rTo^{\subseteq} (\perf(P)^{[\ell-1]},\,
  w) \rTo (\perf(P)^{[\ell-1]}, h_{[\ell]})\] induces a fibration
  sequence of $K$\nbd-theory spaces, by the Fibration Theorem
  \cite[1.6.4]{W-Ktheory}. We have to prove that the $K$\nbd-theory of
  its base is homotopy equivalent, via $\check \Gamma \circ
  \theta_\ell$, to~$K(R)$, where we have written~$\theta_\ell$ for the
  $\ell$th twist functor $Y \mapsto Y(\ell)$.

  First note that $\psi_{-\ell}$ restricts to an exact functor
  \[\psi_{-\ell} \colon \perf(R) \rTo (\perf(P)^{[\ell-1]},\,
  h_{[\ell]}) \ .\] We have to check that $\check \Gamma \circ
  \psi_j$, when applied to a complex of the form~$\psi_{-\ell} (C)$,
  produces an acyclic chain complex for $0 \leq j < \ell$. But for $j$
  in this range we have $-n_P \leq -\ell + j < 0$ so that $E_P(-\ell +
  j) = 0$. It follows that
  \[\check \Gamma \circ \theta_j \circ \psi_{-\ell} (C) \iso \check
  \Gamma C(-\ell+j)\] is acyclic by Lemma~\ref{lem:Gamma_of_psi_k}.

  Now the composition $(\check \Gamma \circ \theta_\ell) \circ
  \psi_{-\ell} \iso \check \Gamma \psi_0$ is weakly equivalent to the
  $n$th suspension endo-functor $C \mapsto C[n]$ of~$\Ch(R)$, by
  Lemmas~\ref{lem:susp_limit} and~\ref{lema:con_psi0}, hence $\check
  \Gamma \circ \theta_\ell$ induces a surjection on homotopy groups of
  $K$\nbd-theory spaces.

  According to Lemmas~\ref{lemma:susp_is_sigma}
  and~\ref{lem:gamma_psi0_gamma} the composition $\psi_0 \circ \check
  \Gamma$ is connected to the $n$th suspension functor by a chain of
  $h_{[0]}$\nbd-equivalences depicted
  \[Y[n] \rTo^\simeq \sigma Y \lTo \psi_0 \circ \check \Gamma (Y)\]
  (where the first map is actually a weak equivalence). Replacing~$Y$
  by its $\ell$th twist~$Y(\ell)$ yields a chain of $h_{[0]}$\nbd-equivalences
  \begin{equation}
    \label{eq:h0-eq}
    \theta_\ell Y[n] \rTo^\simeq \sigma Y(\ell) \lTo \psi_0 \circ \check \Gamma
    \circ \theta_\ell (Y) \ .
  \end{equation}
  Twisting by~$-\ell$ again leaves us with a chain of natural maps
  \[Y[n] \rTo^\simeq_\mu \theta_{-\ell} \big( \sigma Y(\ell) \big)
  \lTo_\nu \psi_{-\ell} \circ \check \Gamma \circ \theta_\ell (Y) \
  .\] We claim that {\it the maps~$\mu$ and~$\nu$ are
    $h_{[\ell]}$\nbd-equivalences}. This is clear for~$\mu$ since
  $\mu$~is in fact a weak equivalence. As for~$\nu$, given an
  integer~$j$ with $0 \leq j \leq \ell$ we have to prove that
  application of~$\check \Gamma \circ \theta_j$ to the map~$\nu$
  produces a quasi-isomorphism of $R$\nbd-module chain complexes. For
  $0 \leq j < \ell$ this is true since both source and target of the
  resulting map of chain complexes are acyclic. Indeed, $\check \Gamma
  \circ \theta_j \big( \theta_{-\ell} \big( \sigma Y(\ell) \big) \big)
  \simeq \check \Gamma Y(j)[n] \simeq 0$ since $Y$~is an object
  of~$\perf(P)^{[\ell-1]}$, while $\check \Gamma \circ \theta_j \circ
  \psi_{-\ell} \circ \check \Gamma \circ \theta_\ell (Y) = \check
  \Gamma \circ \psi_{j-\ell} \circ \check \Gamma \circ \theta_\ell (Y)
  \simeq 0$ by Lemma~\ref{lem:Gamma_of_psi_k}, applied to the chain
  complex $C = \check \Gamma \circ \theta_\ell (Y)$, since
  $E_P(j-\ell) = 0$. For $j = \ell$ note that $\theta_\ell (\nu)$ is
  the map $\sigma Y(\ell) \lTo \psi_0 \circ \check \Gamma Y(\ell)$
  which is an $h_{[0]}$\nbd-equivalence according to
  Lemma~\ref{lem:gamma_psi0_gamma} (applied to~$Y(\ell)$ instead
  of~$Y$).

  We have thus verified the claim. But this means that $n$th
  suspension and $\psi_{-\ell} \circ \check
  \Gamma \circ \theta_\ell$ induce homotopic maps on the
  $K$\nbd-theory space of $(\perf(P)^{[\ell-1]}, h_{[\ell]})$ so
  that $\check \Gamma \circ \theta_\ell$ induces a map of
  $K$\nbd-theory spaces which is injective on homotopy groups.

  This proves that $K(\perf(P)^{[\ell-1]}, h_{[\ell]}) \simeq K(R)$
  via the functor $\check \Gamma \circ \theta_\ell$. The resulting
  fibration sequence
  \[\tilde K(P,\, [\ell]) \rTo K(P,\, [\ell-1]) \rTo^{\check \Gamma
    \circ \theta_\ell} K(R)\] has a section up to homotopy and up to
  sign induced by~$\psi_{-\ell}$ as the composition $\check \Gamma
  \circ \psi_\ell \circ \psi_{-\ell}$ is weakly equivalent to $n$th
  suspension, just as argued above.
\end{proof}

We are now in a position to return to the main theorem of the paper:

\begin{main}
  Let $P \subseteq \bR^n$ be an $n$\nbd-dimensional lattice polytope,
  and let $n_P$ be the number of distinct integral roots of its
  \textsc{Ehrhart\/} polynomial~$E_P(x)$. Let $R$ be a ring with unit;
  suppose that $R$~is commutative, or else left \textsc{noether}ian.
  Then there is a homotopy equivalence of $K$\nbd-theory spaces
  \[K(P) \simeq K(R)^{1+n_P} \times K(P, [n_P])\] where $K(R)$~denotes
  the $K$\nbd-theory of the ring~$R$, and $K(P, [n_P])$ denotes the
  $K$\nbd-theory of those perfect complexes $Y \in \perf(P)$ which
  satisfy $\check \Gamma \big( Y(j) \big) \simeq 0$ for $0 \leq j \leq
  n_P$.  If $n_P = 0$ this expresses the tautological splitting $K(P)
  \simeq K(R) \times \tilde K(P)$ where $\tilde K(P) = K(P,\, [0])$.
\end{main}

\begin{proof}
  This follows by concatenating the homotopy equivalences from
  Propositions~\ref{prop:reduced_ktheory} and~\ref{prop:split_again}
  for $\ell = 1,\, 2,\, \cdots,\, n_p$.
\end{proof}

\subsection{Algebraic $K$-theory of projective space}

\begin{theorem}
  \label{thm:k_of_Pn}
  Let $\Delta^n$ be an $n$\nbd-dimensional simplex with
  volume~$1/n!$. Then $n_{\Delta^n} = n$ and $K(\Delta^n, [n]) \simeq
  *$ so there is a homotopy equivalence
  \[K(\Delta^n) \simeq K(R)^{n+1} \ .\]
\end{theorem}

Let us remark first that an $n$\nbd-dimensional simplex with
volume~$1/n!$ can be transformed, by integral translation and a linear
map in~$GL_n(\bZ)$, into the standard simplex with vertices
$0,\,e_1,\, \cdots,\, e_n \in \bR^n$. Up to isomorphism, the algebraic
data associated to~$\Delta^n$ does not change so that we may assume
$\Delta^n$ to be a standard simplex to begin with. Its
\textsc{Ehrhart\/} polynomial is $E_{\Delta^n}(x) = (x+1)(x+2) \cdots
(x+n)/n!$ which has precisely $n_{\Delta^n} = n$ integral roots.

In case of a commutative ring~$R$ we have $X_{\Delta^n} =
\mathbb{P}^n_R$, projective $n$\nbd-space over~$R$, and the splitting
of Theorem~\ref{thm:k_of_Pn} corresponds to the known splitting of
$K$\nbd-theory of projective $n$\nbd-space which in turn is a special
case of \textsc{Quillen\/}'s ``projective bundle'' theorem in
$K$\nbd-theory applied to the trivial vector bundle of rank~$n+1$ over
the affine scheme $\mathrm{Spec}\, R$ \cite[\S8.2]{Quillen-K}~\cite[4.1]{TT-K}.

\begin{proof}[Proof of~\ref{thm:k_of_Pn}]
  It is enough to prove the following assertion:
  \begin{assertion}
    \label{ass:generators}
    Let $Y \in \hco(\Delta^n)$ be such that $\check \Gamma \big(
    Y(\ell) \big)$ is acyclic for all~$\ell$ with $0 \leq \ell \leq
    n$. Then the chain complexes~$Y^F$ are acyclic for all $F \in
    F(\Delta^n)_0$.
  \end{assertion}
  \noindent For then the map $Y \rTo 0$ in~$\hco(\Delta^n)$ gives a
  weak equivalence of endo-functors of $\perf(P)^{[n]}$ from the
  identity to the zero functor. Consequently, the identity map of
  $K(\Delta^n,\, [n])$ is null homotopic so that $K(\Delta^n,\, [n])
  \simeq *$. The theorem now follows from the splitting
  result~\ref{thm:main} and the fact that $n_{\Delta^n} = n$.

  The above assertion roughly states that the sheaves
  $\mathcal{O}(k)$, $0 \leq k \leq n$, generate the derived category
  of~$\hco(P)$. This point of view has been pursued, in a model
  category context, in \cite[3.3.5]{H-colocal}. We sketch the argument
  for the reader's convenience.

  Suppose $Y \in \hco(\Delta^n)$ has the property that all the
  structure maps $Y^F \rTo Y^G$ are quasi-isomorphisms, for all pairs
  $F \subseteq G$ of non-empty faces of~$\Delta^n$. For fixed $F \in
  F(\Delta^n)_0$ the structure maps then induce a chain of
  quasi-isomorphisms of diagrams
  \[Y \rTo \con Y^{\Delta^n} \lTo \con Y^F \ .\] By
  Proposition~\ref{prop:cech_invariant} and Lemma~\ref{lem:susp_limit}
  we obtain quasi-isomorphisms of chain complexes of $R$\nbd-modules
  \[\check \Gamma Y \rTo^\simeq \check \Gamma \big( \con Y^{\Delta^n}
  \big) \lTo^\simeq \check \Gamma \big( \con Y^F \big) \lTo^\simeq
  Y^F[n] \ .\] So if in addition $\check \Gamma (Y) = \check \Gamma
  \big( Y(0) \big)$ is acyclic we know that~$Y^F[n]$ and hence~$Y^F$
  is acyclic as well. --- It is thus sufficient to prove the following
  claim:
  \begin{assertion}
    \label{ass:structure_maps}
    Let $Y \in \hco(\Delta^n)$ be such that $\check \Gamma \big(
    Y(\ell) \big)$ is acyclic for all~$\ell$ with $0 \leq \ell \leq
    n$. Then the structure maps $Y^F \rTo Y^G$ are quasi-isomorphisms
    for all pairs $F \subseteq G$ of non-empty faces of~$\Delta^n$.
  \end{assertion}
  It is in fact enough to consider structure maps of the form $Y^F
  \rTo Y^{v \vee F}$ for $F \in F(\Delta^n)_0$ and a vertex~$v$
  of~$P$.

  As remarked above, $\Delta^n$~is isomorphic to a standard
  $n$\nbd-simplex with vertices $0,\, e_1,\, \cdots,\, e_n \in \bR^n$;
  the isomorphism can be chosen to map any vertex of~$\Delta^n$
  to~$0$. In view of this symmetry it is enough to prove the
  following:
  \begin{assertion}
    \label{ass:special_structure_map}
    Suppose that $\Delta^n$ is a standard $n$\nbd-simplex, and suppose
    that $Y \in \hco(\Delta^n)$ is such that $\check \Gamma \big(
    Y(\ell) \big)$ is acyclic for all~$\ell$ with $0 \leq \ell \leq
    n$. Then for every face~$F$ of~$\Delta^n$ not containing~$0$ the
    structure map $Y^F \rTo Y^{0 \vee F}$ is a quasi-isomorphism.
  \end{assertion}
  This assertion is proved by
  induction on~$n$, the case $n=0$ being trivial as $\pre(\Delta^0) =
  \hco(\Delta^0) = \Ch (\RMod)$.

  So let $n > 0$. For every face~$F$ of~$\Delta^n$ there is an obvious
  inclusion of sets $kF+T_F \subseteq (k+1)F+T_F$, $0 \leq k < n$,
  which is an equality if and only if $0 \in F$. Hence we have
  corresponding maps $\mathcal{O}(k) \rTo \mathcal{O}(k+1)$ and $Y(k)
  \rTo Y(k+1)$ which are identities if $0 \in F$. We obtain short
  exact sequences in $\pre(\Delta^n)$
  \begin{equation}
    \label{eq:short_exact_sheaves}
    0 \rTo Y(k) \rTo Y(k+1) \rTo Z(k+1) \rTo 0
  \end{equation}
  where $Z(k+1)$ is, {\it a
    priori}, simply a name for the cokernel. However, since taking
  cokernels commutes with tensor products we see that $Z(k+1)$ is indeed
  the $k$th twist of $Z(1) = \mathrm{coker}\, \big( \mathcal{O}(0)
  \rTo \mathcal{O}(1) \big)$. As the functor $\check \Gamma$ preserves
  short exact sequences we conclude that $\check \Gamma Z(k+1)$ is
  acyclic for $0 \leq k < n$ since $\check \Gamma Y(k)$ and $\check
  \Gamma Y(k+1)$ are acyclic in this range by hypothesis.

  We now have to analyse the diagram~$Z(k+1)$ in more detail. If $0
  \in F$ the map $Y(k)^F \rTo Y(k+1)^F$ is the identity, as remarked
  above, so $Z(k+1)^F = 0$. --- Suppose now that $F$~is a face
  of~$\Delta^n$ with $0 \notin F$. There is in fact a maximal face
  $\Delta^{n-1}$ of~$\Delta^n$ not containing~$0$, and~$F$ is a face
  of~$\Delta^{n-1}$. We will argue that $Z(k+1)$, when restricted
  to~$F(\Delta^{n-1})_0$, is naturally an object
  of~$\hco(\Delta^{n-1})$ with $\check \Gamma_{\Delta^{n-1}} Z(k+1)$
  being acyclic for $0 \leq k < n$. By induction this implies:
  \begin{assertion}
    \label{ass:Z1_acyclic}
    The chain complex $Z(1)^F$~is acyclic for all $F \in
    F(\Delta^{n-1})_0$ and hence for all $F \in F(\Delta^n)_0$
  \end{assertion}
  
First let $\bR^{n-1}$ denote the affine hull of~$\Delta^{n-1}$,
  turned into a vector space by distinguishing a lattice point as
  origin. It comes equipped with its own integer lattice $\bZ^{n-1} =
  \bZ^n \cap \bR^{n-1}$. Let $F \in F(\Delta^{n-1})_0$. Then
  \[\big( ((k+1)F +T_F) \cap \bZ^n \big) \setminus \big((kF + T_F) \cap \bZ^n
  \big) = ((k+1)F + T^{\bR^{n-1}}_F) \cap \bZ^{n-1} \ ,\] the barrier
  cones on the left being computed in~$\bR^n$, the barrier cone on the
  right being computed in~$\bR^{n-1}$. This translates into an
  isomorphism
  \begin{equation}
    \label{eq:twisting_iso}
    \mathrm{coker}\, \big(\mathcal{O}_{\Delta^n} (k)^F \rTo
    \mathcal{O}_{\Delta^n} (k+1)^F \big) \iso \mathcal{O}_{\Delta^{n-1}}
    (k+1)^F \ .
  \end{equation}
  By considering $k = -1$ we obtain from this an algebra
  isomorphism
  \begin{equation}
    \label{eq:algebra_iso}
    A_F^{\Delta^n} / \mathcal{O}_{\Delta^n} (-1)^F =
    \mathcal{O}_{\Delta^n}(0)^F / \mathcal{O}_{\Delta^n} (-1)^F \iso
    A^{\Delta^{n-1}}_F
  \end{equation}
  and thus an algebra epimorphism $A^{\Delta^n}_F \rTo
  A^{\Delta^{n-1}}_F$. The isomorphism displayed
  in~(\ref{eq:algebra_iso}) is used to equip~$Z(k+1)^F$ with a natural
  $A_F^{\Delta^{n-1}}$\nbd-module structure while the
  isomorphism~(\ref{eq:twisting_iso}) is used to verify that twisting
  with respect to~$\Delta^{n}$ and~$\Delta^{n-1}$, respectively, is
  compatible. A straightforward $5$\nbd-lemma argument shows that
  $Z(k)$, when considered as an object of~$\pre(\Delta^{n-1})$, is
  indeed a homotopy sheaf. Finally, from the definition of \textsc{\v
    Cech\/} complexes it follows that the chain complexes $\check
  \Gamma_{\Delta^n} \big(Z(k+1)\big)$ and $\check
  \Gamma_{\Delta^{n-1}} \big( Z(k+1)|_{F(\Delta^{n-1})_0} \big)$ agree
  up to re-indexing by~$1$. In total, this means that $Z(k+1) \in \hco
  (\Delta^{n-1})$ satisfies the induction hypotheses as claimed. We
  have thus verified~\ref{ass:Z1_acyclic}.

  From the short exact sequence~(\ref{eq:short_exact_sheaves}),
  restricted to $F$\nbd-components, we infer that the map $\sigma_F
  \colon Y(0)^F \rTo Y(1)^F$ is a quasi-isomorphism of chain
  complexes, hence so is the map from $Y(0)^F$ to the colimit of the
  infinite sequence
  \[Y(0)^F \rTo_{\sigma_F}^\simeq Y(1)^F \rTo^\iso Y(0)^F
  \rTo_{\sigma_F}^\simeq Y(1)^F \rTo^\iso Y(0)^F \rTo^\simeq \cdots \
  .\] Here every second map is a fixed isomorphism between $Y(0)^F$
  and~$Y(1)^F$. It is not difficult to see that the colimit of this
  sequence is isomorphic to $A_{0 \vee F} \tensor_{A_F} Y^F$ (this
  follows from the fact that the cone $T_{0 \vee F}$ is obtained from
  the cone~$T_F$ by forming \textsc{Minkowski\/} sum with a single ray
  spanned by the negative of a vector in~$T_F \cap \bZ^n$). Now the
  composite
  \[Y^F \rTo A_{0 \vee F} \tensor_{A_F} Y^F \rTo Y^{0 \vee F}\] is a
  structure map of~$Y$, and both constituent maps are
  quasi-isomorphisms: the first by what we have just shown, the second
  by the stipulation that $Y$~be a homotopy sheaf. In total, we have
  verified that the structure map $Y^F \rTo Y^{0 \vee F}$ is a
  quasi-isomorphism. We are done.
\end{proof}


\begin{thebibliography}{H{\"u}t09b}

\bibitem[Bar08]{Barvinok-latticepoints}
Alexander Barvinok, \emph{Integer points in polyhedra}, Zurich Lectures in
  Advanced Mathematics, European Mathematical Society (EMS), Z\"urich, 2008.

\bibitem[Ful93]{Fulton:TVs}
William Fulton, \emph{Introduction to toric varieties}, Annals of Mathematics
  Studies, vol. 131, Princeton University Press, Princeton, NJ, 1993, The
  William H. Roever Lectures in Geometry.

\bibitem[Har77]{Ha-AG}
Robin Hartshorne, \emph{{Algebraic geometry}}, {Graduate Texts in Mathematics.
  52. New York-Heidelberg-Berlin: Springer-Verlag}, 1977.

\bibitem[H{\"u}t04]{H-finiteness}
Thomas H{\"u}ttemann, \emph{{Finiteness of total cofibres}}, {\it K}-Theory
  \textbf{31} (2004), 101--123.

\bibitem[H{\"u}t09]{H-nonlin}
\bysame, \emph{{$K$-theory of non-linear projective toric varieties}}, Forum
  Math. \textbf{21} (2009), no.~1, 67--100.

\bibitem[H{\"u}t10]{H-colocal}
\bysame, \emph{On the derived category of a regular toric scheme}, Geometriae
  Dedicata \textbf{148} (2010), 175--203. 

\bibitem[MS05]{MS-CombAlg}
Ezra Miller and Bernd Sturmfels, \emph{{Combinatorial commutative algebra}},
  {Graduate Texts in Mathematics 227. New York, NY: Springer}, 2005.

\bibitem[Qui73]{Quillen-K}
Daniel Quillen, \emph{{Higher algebraic {$K$}-theory. I.}}, {Algebr.
  {$K$}-Theory I, Proc. Conf. Battelle Inst. 1972, Lect. Notes Math. 341},
  1973, pp.~85--147.

\bibitem[Ros94]{Rosen-K}
Jonathan Rosenberg, \emph{{Algebraic K-theory and its applications}},
  {Graduate Texts in Mathematics. 147. New York, NY: Springer-Verlag}, 1994.

\bibitem[Sta86]{Stanley-popo}
Richard~P. Stanley, \emph{Two poset polytopes}, Discrete Comput. Geom.
  \textbf{1} (1986), no.~1, 9--23.

\bibitem[TT90]{TT-K}
R.W. Thomason and Thomas Trobaugh, \emph{{Higher algebraic {\it K}-theory of
  schemes and of derived categories}}, The {\textsc{Grothendieck}} Festschrift.
  A collection of articles written in honor of the 60th birthday of Alexander
  Grothendieck, Boston, MA: Birkh{\"a}user, 1990, pp.~247--435.

\bibitem[Wal85]{W-Ktheory}
Friedhelm Waldhausen, \emph{{Algebraic {$K$}-theory of spaces}}, {Algebraic
  and geometric topology, Proc. Conf., New Brunswick/USA 1983, Lect. Notes
  Math. 1126}, 1985, pp.~318--419.

\bibitem[Wei95]{Weibel-intro}
Charles~A. Weibel, \emph{{An introduction to homological algebra}}, {Cambridge
  Studies in Advanced Mathematics. 38. Cambridge: Cambridge Univ. Press}, 1995.

\end{thebibliography}

\providecommand{\bysame}{\leavevmode\hbox to3em{\hrulefill}\thinspace}
\providecommand{\MR}{\relax\ifhmode\unskip\space\fi MR }
\providecommand{\MRhref}[2]{%
  \href{http://www.ams.org/mathscinet-getitem?mr=#1}{#2}
}
\providecommand{\href}[2]{#2}

\bigskip
\tiny
\noindent (Version of paper: 16.09.2010, minor corrections 16.12.2010,
01.03.2011, 04.07.2011)

\end{document}